\documentclass[10pt]{amsart}
	\usepackage{times,amsmath,amsbsy,amssymb,amscd,mathrsfs}
\usepackage{slashbox}\usepackage{graphicx,subfigure,epstopdf,wrapfig,chemarrow}
\usepackage{multicol,multirow}
\usepackage{mathtools}
\usepackage[usenames,dvipsnames,svgnames,table]{xcolor}
\usepackage[all]{xy}
\usepackage{wrapfig}
\usepackage{algorithm2e} 
\usepackage{tikz}
\usepackage{pgfplots}

\newcommand{\iprd}[2]{\left( #1 , #2 \right)}

\newcommand{\Tauh}{\mathcal{T}_h}
\newcommand{\TauH}{\mathcal{T}_H}

\newcommand{\vertiii}[1]{{\left\vert\kern-0.25ex\left\vert\kern-0.25ex\left\vert #1 \right\vert\kern-0.25ex\right\vert\kern-0.25ex\right\vert}}

\def\norm#1#2{\left\| #1 \right\|_{#2}}

\newcommand{\x}{\chi}

\def\0{\mbox{\boldmath $0$}}

\newcommand{\nrm}[1]{\left\| #1 \right\|}

\usepackage[numbered]{mcode}
\definecolor{myBlue}{rgb}{0.0,0.0,0.55}
\definecolor{green}{rgb}{0.0,0.7,0.2}
\usepackage[pdftex,colorlinks=true,citecolor=myBlue,linkcolor=myBlue]{hyperref}

\newcommand{\LC}[1]{\textcolor{black}{#1}}

\usepackage{comment,enumerate,multicol,xspace}

  \newcounter{mnote}
  \let\oldmarginpar\marginpar
    \renewcommand\marginpar[1]{\-\oldmarginpar[\raggedleft\footnotesize #1]%
    {\raggedright\footnotesize #1}}

\newtheorem{theorem}{Theorem}[section]
\newtheorem{lemma}[theorem]{Lemma}
\newtheorem{corollary}[theorem]{Corollary}
\newtheorem{proposition}[theorem]{Proposition}

\newtheorem{remark}[theorem]{Remark}

\newcommand{\dd}{\,{\rm d}}

\newcommand{\mc}{\mcode}

\usepackage[margins]{trackchanges}
\renewcommand{\initialsOne}{chen}

\begin{document}
\title[Convergence Analysis of FASD]{Convergence Analysis of the \LC{Fast Subspace Descent Methods} for Convex Optimization Problems}

\author{Long Chen}%
\address{Department of Mathematics, University of California at Irvine, Irvine, CA 92697, USA}%
\email{chenlong@math.uci.edu}%
\author{Xiaozhe Hu}%
\address{Department of Mathematics, Tuffs University, Medford, MA 02155, USA}%
\email{Xiaozhe.Hu@tufts.edu}%
\author{Steven M. Wise}%
\address{Mathematics Department, The University of Tennessee, Knoxville, TN 37996, USA}%
\email{swise1@utk.edu}%

\thanks{The second author was supported by NSF Grant DMS-1620063.}
\thanks{The third author was supported by the NSF Grant DMS-1719854.}


	\begin{abstract}
The \LC{full approximation storage} (FAS) scheme is a widely used multigrid method for nonlinear problems. In this paper, a new framework to design and analyze FAS-like schemes for convex optimization problems is developed. The new method, \LC{the Fast Subspace Descent (FASD) scheme}, which generalizes classical FAS, can be recast as an inexact version of nonlinear multigrid methods based on space decomposition and subspace correction. The local problem in each subspace can be simplified to be linear and one gradient descent iteration \LC{(with an appropriate step size)} is enough to ensure a global linear (geometric) convergence of FASD.
	\end{abstract}

	\maketitle


	\section{Introduction}

Most real-world applications are inherently nonlinear. The design of fast algorithms for the solution or approximate solution of nonlinear equations is of fundamental interest to mathematicians, physicists, biologists, and others.  In this paper, we consider solving nonlinear equations arising from the minimization of a convex functional in the abstract Hilbert space setting.  

The well-known Newton-Raphson method is a traditional and popular approach for solving nonlinear equations.  Basically, Newton's method iteratively finds the approximate solution by linearizing the problem near the current iterate.  In the present case, a linear symmetric positive definite system (the Jacobian system) needs to be solved at each Newton's iteration, and fast linear multigrid (MG) methods are sometimes used as a solver. Practically, each linear problem can be approximately inverted by applying a few multigrid iterations. But, if this is done, the quadratic rate of convergence may be sacrificed.  

One  alternative to Newton's method for solving nonlinear PDE is the nonlinear multigrid method, better known as the \LC{full approximation storage (FAS)} scheme. This method, developed by Brandt~\cite{Brandt.A1977} in \LC{the late 70's (see also~\cite{Brandt;Livne:2011Multigrid})} often converges linearly and with optimal complexity in practice.  Recall that the success of multigrid methods relies on two ingredients: 1) high frequency components of the error will be damped by smoothers; and 2) low frequency components of the error can be approximated well on a coarse grid. The smoother used in FAS is usually the nonlinear Gauss-Seidel smoother, which solves many  small-sized (typically  1-D) nonlinear problems on small patches of the mesh.  
For the coarse grid problem, the FAS method uses the full approximation rather than the standard defect, which makes it essentially different from linear MG methods. Due to its high efficiency, the FAS method has been applied to many nonlinear PDE problems, such as in~\cite{Henson.V2005,Spitaleri.R2000,Wise.S;Kim.J;Lowengrub.J2007,Hu;Wise;Wang;Lowengrub:2009Stable,Nash.S2000,Yavneh;Dardyk:2006Multilevel,Huang2016}.  

Although FAS is quite successful in practice, its theoretical analysis is limited.  In~\cite{Hackbusch.W.1985a}, Hackbusch considered nonlinear MG methods for general nonlinear problems.  By imposing conditions on the nonlinear operators and their derivatives, together with standard smoothing and approximation properties, he was able to show that the FAS converges in a sufficiently small neighborhood of the solution on a fine enough mesh.  Moreover, the number of smoothing steps needs to be sufficiently large, and at least the W-cycle should be used.  Later in~\cite{Reusken.A1988a,Reusken.A1988b}, Reusken considered FAS for a class of second order elliptic boundary value problems with mild nonlinearity.   Within this nice class of nonlinear problems, he was able to show the convergence of FAS under weaker assumptions on the nonlinear operators.  We want to mention that the proofs in their work are based on the linearization of the FAS iterations, and the rate of convergence is in some sense local. For example, in \cite{Reusken.A1988b}, Reusken showed that the V-cycle FAS converges locally in a ball with radius shrinking from coarse to fine levels.



In this paper we consider a special class of nonlinear equations that can be viewed as Euler equations of certain convex objective functions. The convergence of MG methods for convex optimization problems has been studied in~\cite{Tai.X;Xu.J1999,Tai;Xu:2001Global} under the framework of subspace correction methods~\cite{Xu.J1992}.  In~\cite{Tai;Xu:2001Global}, Tai and Xu considered some unconstrained convex optimization problems and developed global and uniform convergence estimates for a class of subspace correction iterative methods.  Their approach is based on an abstract space decomposition which is assumed to satisfy the so-called stable decomposition property and strengthened Cauchy Schwarz inequality.  We  point out that in each subspace, the original objective function is used, which is, strictly speaking, naturally defined on the finest level. Furthermore, the local problem should be solved exactly, which is more expensive than what is required in the FAS scheme.


We shall borrow the theoretical framework established in~\cite{Tai;Xu:2001Global} to analyze a hybrid of the FAS and subspace correction methods, what we will call the \LC{\emph{Fast Subspace Descent} (FASD) method}. \LC{In contrast to} the subspace correction method considered in~\cite{Tai;Xu:2001Global}, in which an exact subspace solver is used, we recast FASD as a subspace correction method with an inexact subspace solver, which reduces the computational cost significantly. \LC{In particular, we show that one step of preconditioned gradient descent iteration in each subspace is good enough to guarantee the global convergence.}

\LC{Many other adaptation of FAS to optimization methods can be found in \cite{Gelman;Mandel:1990multilevel,Lewis;Nash:2005problems,Nash.S2000,Gratton;Sartenaer;Toint:2008Recursive} for either  line search-based recursive or trust region-based recursive algorithms. Only plain convergence is established in these work. Here we shall prove a linear convergence for strongly convex optimization problems.}
 
We establish the convergence of the algorithm in the framework of subspace corrections~\cite{Tai;Xu:2001Global}. We first show that, with a one dimensional line search approach, the FASD method converges globally and uniformly under the standard assumptions on the space decomposition. In addition, we borrow some techniques from the optimization literature~\cite{Nesterov:2013Introductory} in order to properly handle the inexactness of the local solver used in FASD.  We introduce a fixed step size to guarantee that the objective function is decreasing globally. For the analysis of original FAS method, which is obtained from the new FASD method via a simple modification, we impose an additional approximation property of the subspace problems and show that FASD converges globally and uniformly. We emphasize that our work represents not only a theoretical advance for the convergence analysis of FAS-type schemes, but also is algorithmically simpler, and even more flexible, than the original FAS. We show that, both theoretically and numerically, each local nonlinear problem can be approximated by a linear problem, and, consequently, the computational cost is reduced significantly.

\LC{The paper is organized as follows. In Section~\ref{sec:problem-assumption}, we present the optimization problem, with its associated Euler equation, in a general Hilbert space framework. We conclude the section with the assumptions on the space decomposition.  The successive subspace optimization (SSO) method are recalled in Section~\ref{sec:SSO}.  The convergence analysis of SSO, based on slightly weaker assumptions compared with~\cite{Tai;Xu:2001Global}, is presented in the same section.  The main global and uniform convergence analyses for FASD with the exact line search and approximate (quadratic) line search are derived in Sections~\ref{sec:FASD} and~\ref{sec:FASD-approx-line}, respectively. The original FAS method is analyzed in Section~\ref{sec:Original-FAS}. In Section~\ref{sec:numerics}, an application problem is considered.}

	\section{Problem and Assumptions}
	\label{sec:problem-assumption}

Given an energy $E(v)$ defined on a Hilbert space $\mathcal V$ equipped with inner product $(\cdot, \cdot)_{\mathcal{V}}$ and norm $\| \cdot \|_{\mathcal{V}}$, we consider the following minimization problem:
	\begin{equation}
	\label{intro:main-opt}
u = \mathop{\rm argmin}_{v\in \mathcal V} E(v).
	\end{equation}
We now make some assumptions that guarantee that the minimizer exists and is unique.

	\subsection{Assumptions on the Energy}

We assume that the energy functional $E(\, \cdot \, ): \mathcal V\to\mathbb{R}$ is Fr\'{e}chet differentiable for all points $v\in \mathcal V$. For each fixed $v\in \mathcal V$, $E'(v):\mathcal V\to\mathbb{R}$ is the continuous linear functional equal to the first Fr\'{e}chet derivative at $v$. We further impose the following assumptions on the energy:

	\begin{itemize}
 	\item[(E1)](Strong convexity):
There is a constant $\mu>0$ such that
	\begin{equation}
\mu \norm{w-v}{\mathcal{V}}^2\le  {\langle E'(w) -E'(v)  , w -v\rangle},
	\label{assmp-L4} 
	\end{equation}
for all $v,w\in \mathcal V$, where $\langle \, \cdot \, , \,  \cdot \,  \rangle$ is the duality pairing between $\mathcal{V}'$ and $\mathcal{V}$.

      \item[(E2)](Lipschitz continuity of the first order derivative): 
For fixed $u_0\in \mathcal{V}$, there exists a constant $L$ such that, for all $v, w\in \mathcal{B}:= \left\{v \in  \mathcal V  \ \middle| \ E(v)\le E(u_0)\right\}$, 
	\begin{equation}
	\label{eqn:Lipschitz}
\| E'(w) - E'(v) \|_{\mathcal{V}'} \leq L \| w - v \|_{\mathcal{V}},
	\end{equation}
where 
	\[
\| f \|_{\mathcal{V}'} := \sup_{\substack{v \in \mathcal{V} \\ \|v\|_{\mathcal V} =1}} \langle f, v \rangle  = \sup_{v\in\mathcal{V}\setminus\{0\}} \frac{\langle f,v\rangle}{\nrm{v}_{\mathcal{V}}}. 
	\] 
	\end{itemize}
	
\LC{Other authors, for example Ciarlet~\cite{Ciarlet89}, use the term \emph{elliptic} for the property in assumption (E1). We should also point out that assumption (E1) is equivalent to the property that the derivative is \emph{strongly  monotone}~\cite{Atkinson09}.}

The following results are classical, and the proof, which is skipped for the sake of brevity, can be found in \cite[p. 35]{Ekeland;Temam:1976Convex}, \cite[Thm.~8.2-2]{Ciarlet89}, or \cite[Thm.~3.3.13]{Atkinson09}.

	\begin{theorem}
	\label{thm:convex&coercive}
If $E$ satisfies assumption (E1), then, for all $w,v\in \mathcal{V}$
	\begin{equation}
E(w) - E(v) \ge \langle E'(v),w-v\rangle + \frac{\mu}{2}\nrm{w-v}_{\mathcal{V}}^2.
	\label{ineq:strict-convexity}
	\end{equation}
Consequently, $E$ is strongly convex and coercive. Furthermore, there is a unique element $u\in \mathcal V$ with the property that
	\[
E(u) \le E(v), \quad \forall \, v \in \mathcal V, \quad \mbox{and} \quad E(u) < E(v), \quad \forall \, v\ne u,
	\]
and this \emph{global minimizer} satisfies Euler equation
	\begin{equation}
	\label{eqn:Euler}
\langle E'(u), w \rangle = 0, \quad \forall \ w \in \mathcal V.
	\end{equation}
	\end{theorem}

The strong convexity and the Lipschitz continuity imply the following estimates:
	\begin{lemma} 
	\label{lm:dE}
Suppose $E$ satisfies assumptions (E1) and (E2). For all $v,w\in \mathcal{B}$, 
	\begin{equation*}
\mu\nrm{w-v}^2_{\mathcal{V}} \le \langle E'(w) -E'(v) , w-v\rangle  \le L\nrm{w-v}^2_{\mathcal{V}}.
	\end{equation*}
Furthermore the lower bound holds for all $v,w\in \mathcal V$.
	\end{lemma}
	\begin{proof}
The lower bound is just assumption (E1). To get the upper bound, observe that (E2) implies that, for all $w,v\in\mathcal{B}$, and for any $z\in \mathcal{V}$,
	\[
\left|\langle E'(w)-E'(v), z \rangle\right|  \le \nrm{E'(w)-E'(v)}_{\mathcal{V}'} \nrm{z}_{\mathcal{V}} \le  L \nrm{w-v}_{\mathcal{V}} \nrm{z}_{\mathcal{V}}.	
	\]
Setting $z = w-v$ gives the desired inequality.
	\end{proof}
	
	\begin{proposition}
	\label{prop:B-convex}
If $E$ satisfies (E1), the \LC{sublevel set} $\mathcal{B}$ is convex.
	\end{proposition}
	\begin{proof}
Suppose that $v,w\in\mathcal{B}$. Then $E(w)\le E(u_0)$ and $E(v)\le E(u_0)$. Since $E$ is strictly convex, for any $t\in[0,1]$,
	\[
E(u_0) \ge (1-t)E(w) + tE(v) \ge E((1-t)w+t v) .	
	\] 
Thus, $(1-t)w+tv \in\mathcal{B}$, for any $t\in[0,1]$.
	\end{proof}

Now, we consider the relation between the energy and the norm centered at the minimizer. The following estimates can be easily proved using Taylor's theorem with integral remainder; see, e.g.~\cite{Nesterov:2013Introductory}.

	\begin{lemma}[Quadratic Energy Trap]
	\label{lm:energyatu-V}
Suppose $E$ satisfies assumptions (E1) and (E2). For all $v, w \in \mathcal{B}$, 
	\begin{equation} 
	\label{ine:energy}
\frac{\mu}{2} \nrm{w-v}^2_{\mathcal{V}}  + \langle E'(v), w-v \rangle \le E(w)-E(v) \le  \langle E'(v), w-v \rangle + \frac{L}{2}\nrm{w-v}_{\mathcal{V}}^2.
	\end{equation}
Furthermore the lower bound holds for all $v,w\in \mathcal V$. 
In addition, suppose $u \in \mathcal{B}$ is the minimizer of $E$, then for all $w \in \mathcal{B}$, 
	\begin{equation}
	\label{ine:energy-bound}
\frac{\mu}{2} \nrm{w-u}^2_{\mathcal{V}}  \le E(w)-E(u) \le   \frac{L}{2}\nrm{w-u}_{\mathcal{V}}^2.
	\end{equation}
Again the lower bound holds for all $w\in \mathcal V$.
	\end{lemma}
	

Based on assumption (E1), the upper bound can be replaced by a norm of the gradient. Since the proof is less standard, we include it here. 
%

	\begin{lemma}
	\label{lm:dE-energy-V}
Suppose that $E$ satisfies assumption (E1) and $u \in \mathcal{V}$ is the minimizer of $E$; then for all $v \in \mathcal{V}$, 
	\begin{equation}
	\label{ine:energy-dE-V} 
0 \le E(v) - E(u)  \leq \frac{1}{2\mu} \| E'(v) \|_{\mathcal{V}'}^2.
	\end{equation}
	\end{lemma}
	\begin{proof}
Fix the point $v\in\mathcal{V}$. Now, for any $w\in \mathcal{V}$, using the lower bound of~\eqref{ine:energy}, we have
	\begin{equation*}
E(w) \geq E(v) +  \langle E'(v), w-v \rangle + \frac{\mu}{2} \| w -v \|_{\mathcal{V}}^2 =: g(w).
	\end{equation*}
For fixed $v \in \mathcal{V}$, the minimizer of $g(w)$ is $w^* := v - \frac{1}{\mu} \mathfrak{R}E'(v)$, where $\mathfrak{R}E'(v)$ is the Riesz representation {in $\mathcal{V}$} of $E'(v)$. Therefore, 
\begin{align*}
E(w) \geq g(w) \geq g(w^*) = E(v) - \frac{1}{2\mu} \| \mathfrak{R}E'(v) \|_{\mathcal{V}}^2= E(v) - \frac{1}{2\mu} \| E'(v) \|_{\mathcal{V}'}^2.
\end{align*}
Then~\eqref{ine:energy-dE-V} is obtained by letting $w = u$ in the above inequality.  
	\end{proof}

We shall often use the following simple variant of Lemma~\ref{lm:energyatu-V}.
	\begin{lemma}[Convexity of Energy Sections]
	\label{lm:energyatu-variant}
Suppose that $E$ satisfies (E1) -- (E2), $\xi\in \mathcal{B}$ is arbitrary, and $\mathcal{W}\subseteq \mathcal{V}$ is a \LC{closed} subspace.  Define the energy section
	\[
J(w) := E(\xi + w),\quad \forall \ w\in \mathcal{W}.
	\]
Then $J:\mathcal{W}\to\mathbb{R}$ is differentiable, strongly convex, and there exists a unique element $\eta\in \mathcal{W}$ such that $\xi+\eta\in \mathcal{B}$, $\eta$ is the unique global minimizer of $J$, and
	\[
\langle E'(\xi+\eta),w\rangle = \langle J'(\eta),w\rangle = 0 , \quad \forall \ w \in \mathcal{W}.
	\]
Furthermore, for all $w \in  \mathcal{W}$ with $w+\xi\in\mathcal{B}$, 
	\[
\frac{\mu}{2} \nrm{w-\eta}^2_{\mathcal{V}}  \le J(w)-J(\eta) = E(\xi+w) - E(\xi+\eta) \le   \frac{L}{2}\nrm{w-\eta}_{\mathcal{V}}^2.
	\]
The lower bound holds for any $w\in\mathcal{W}$, without restriction.
	\end{lemma}

The ratio $L/\mu$ is called the condition number of the derivative $E'$; see \cite[page 63]{Nesterov:2013Introductory}. The rate of convergence of iterative methods for solving \eqref{intro:main-opt} usually depends on the condition number. Here we assume $L/\mu$ is uniformly bounded, as long as we remain in $\mathcal{B}$. Then the Riesz map $\mathfrak{R}: \mathcal V' \to \mathcal V$  can be used as a preconditioner and the corresponding preconditioned gradient descent method will converge~\cite{Feng;Salgado;Wang;Wise:2017Preconditioned}. 

Implementing preconditioned gradient descent methods in $\mathcal V$ requires the computation of the Riesz map $\mathfrak{R}$ which is equivalent to inverting a symmetric positive definite (SPD) operator (an SPD matrix of dimension $\dim \mathcal V \times \dim \mathcal V$ when $\dim \mathcal V<+\infty$). Of course we can also use multilevel methods to compute $\mathfrak{R}$ and use steepest descent, nonlinear conjugate gradient, or Newton method as the outer iteration. In the following, we shall provide optimization methods that only require computing inverses with much smaller sizes. 

	\subsection{Assumptions on the Space Decomposition} Suppose that
	\[
\mathcal V = \mathcal V_1 + \mathcal V_2 + \cdots + \mathcal V_N, \qquad  \mathcal V_i\subseteq  \mathcal V, \quad i=1,\ldots, N,
	\]
is a space decomposition of $\mathcal V$ \LC{using closed subspaces} $\mathcal V_i$ for $i=1,2,\ldots, N$. We shall use the following assumptions on the space decomposition.

	\begin{itemize}

	\item[(SS1)] (Stable decomposition:)
There is a constant $C_A > 0$, such that, for every $v\in \mathcal V$, there exists $v_i\in \mathcal V_i$, $i=1, \cdots, N$, with the property that 
	\[
v = \sum _{i=1}^N v_i, \quad \mbox{and} \quad \sum_{i=1}^N \|v_i\|_\mathcal{V}^2 \leq C_A^2\|v\|_\mathcal{V}^2.
	\]

	\item[(SS2)] (Strengthened Cauchy Schwarz inequality:)
There is a constant $C_S >0$, such that, for any $w_{i,j}\in \mathcal B$,  $u_i \in \mathcal V_i$, $v_i\in \mathcal V_i$, with $w_{i,j}+u_i\in \mathcal B$,
	\begin{align*}
\sum_{i=1}^{N}\sum_{j=i+1}^N \langle  E'(w_{i,j}+u_j)-E'(w_{i,j}) , v_i \rangle 
\le \ C_S \left (\sum_{i=1}^{N}\nrm{u_i}_\mathcal{V}^2\right )^{1/2}\left (\sum_{i=1}^{N}\nrm{v_i}_\mathcal{V}^2\right )^{1/2}.
	\end{align*} 
	\end{itemize}

\LC{When $E''$ exists and is continuous, by the mean value theorem and standard Cauchy Schwarz inequality
	\[
\langle  E'(w_{i,j}+u_j)-E'(w_{i,j}) , v_i \rangle = \langle  E''(\xi_i) u_j , v_i \rangle \leq \| E''(\xi_i)\|\|u_j\|_{\mathcal V}\|v_i\|_{\mathcal V}.
	\]
Thus a naive verification of (SS2) could use the constant $C_S = L N$, which would be large if $N$ is large. When the inner product induced by the Hessian $E''(\xi_i)$ is spectrally equivalent to an ${\mathcal V}$-inner product, a better constant $C_S$, which is independent of $N$, can be obtained. This is reason we call it Strengthened Cauchy Schwarz inequality. 
}

We note that the constant $C_S>0$ \LC{can be} related to the Lipschitz constant in assumption (E2). Unless $E$ was quadratic, we could not assume in general that the Strengthened Cauchy Schwarz inequality would hold without restriction to the bounded sublevel set $\mathcal{B}$, as indicated in assumption (SS2).

	\section{Successive Subspace Optimization Methods} \label{sec:SSO}

For $k\geq 0$ and a given approximate solution $u^k\in \mathcal V$, one step of the Successive Subspace Optimization (SSO) method~\cite{Tai:2003convergence} is given in Algorithm~\ref{alg-sso}.

  \begin{algorithm}
\TitleOfAlgo{$u^{k+1} = {\rm SSO}(u^k)$}
$v_0 = u^k$
	\;
	\For{$i=1:N$}{
Define an energy section along $\mathcal{V}_i$:
	\[
J_i(w) := E(v_{i-1} + w), \quad \forall w \, \in\mathcal{V}_i; 
	\]
	\\
Compute the subspace correction: 
	\begin{equation}
e_i = \mathop{\rm argmin}_{w\in \mathcal V_i} J_i(w) ;
	\label{eqn:SSO-correction}
	\end{equation}
	\\
Apply the subspace correction:
	\[
v_i = v_{i-1} + e_i ;
	\]
}
$u^{k+1}=v_N$
	\;
	\medskip
\caption{Successive Subspace Optimization Method.}
	\label{alg-sso}
	\end{algorithm}
	
	\begin{remark}\rm 
Note that $e_i$ computed in \eqref{eqn:SSO-correction} of Algorithm~\ref{alg-sso} is uniquely defined, owing to the strong convexity inherited by the energy section $J_i$. In fact, the correction satisfies
	\[
\langle E'(v_i), w \rangle = \langle E'(v_{i-1} + e_i), w \rangle = \langle J'(e_i), w\rangle = 0, \quad \forall \ w\in \mathcal{V}_i.
	\]
The orthogonality relation satisfied by the corrected approximation, $v_i$, specifically,
	\[
\langle E'(v_i), w \rangle = 0, \quad \forall \ w\in \mathcal{V}_i,
	\]
\LC{is sometimes referred to as the \emph{fundamental orthogonality} (FO) of the solver.}
	\end{remark}

\LC{	
	\begin{remark}\rm 
We point out that, when $\mathcal V_i$ is one-dimensional, then the computation of the subspace correction is identical to a nonlinear Gauss-Seidel method. In fact, the SSO method can be considered as a generalization of the nonlinear Gauss-Seidel methodology.
	\end{remark}
}
We aim to prove a linear reduction of the energy difference for one iteration of the SSO algorithm:
	\begin{equation}
	\label{linearreduction}
E(u^{k+1}) - E(u) \leq \rho ( E(u^{k}) - E(u)),
	\end{equation}
where $u$ is the minimizer of $E$ and $u^{k+1} = {\rm SSO}(u^k)$, with a contraction factor $\rho \in (0,1)$. Ideally $\rho$ is independent of the size of the problem. The algorithm and convergence theory has been developed in~\cite{Tai:2003convergence,Tai;Xu:2001Global} for a convex energy in Banach spaces. For completeness, we include a simplified version for Hilbert space here.

\LC{
We will utilize the following simple result:
	\begin{theorem}
	\label{thm:linear-conv}
Suppose that $\{d_k\}_{k=0}^\infty$, $\{\delta_k\}_{k=0}^\infty$, $\{\eta_k\}_{k=0}^\infty$ are sequences of non-negative real numbers, the first two having the relationship 
	\[
\delta_k = d_k - d_{k+1}, \quad k = 0, 1, 2, \cdots .
	\]
Assume that there are constants $C_L,C_U >0$, independent of $k$, such that
	\[
 C_L \eta_k \le \delta_k  \quad \mbox{and} \quad d_{k+1} \le C_U \eta_k.
	\]
Then
	\begin{equation}
	\label{ineq-linear-conv}
d_{k+1} \le \frac{C_U}{C_L + C_U} d_k, \quad k = 0, 1, 2, \cdots.
	\end{equation}
Consequently $\{d_k\}$ converges monotonically, and (at least) linearly to 0.
	\end{theorem}
	\begin{proof}
Observe that
	\[
d_{k+1} \le C_U \eta_k = \frac{C_U}{C_L} C_L \eta_k \le \frac{C_U}{C_L} \delta_k = \frac{C_U}{C_L}(d_k - d_{k+1}),
	\]
which implies \eqref{ineq-linear-conv}. Proving that $\{d_k\}$ is strictly decreasing to zero is straightforward, and the proof is omitted. 
	\end{proof}
}

We will apply the last result with the following definitions: 
	\begin{equation}
	\label{eqns:d_and_delta}
d_k := E(u^k) - E(u) \quad  \mbox{and} \quad  \delta_k := E(u^k) - E(u^{k+1}).
	\end{equation}
The quantity $d_k$ is the difference between the current energy and the minimum energy, also known as optimality gap, and $\delta_k$ is the energy decrease associated to the ${k+1}$-th iteration. They are connected, as desired, by the trivial identity 
	\[
\delta_k = d_k - d_{k+1}.
	\]
	See Fig.~\ref{fig:delta_k} for an illustration.
We define $\eta_k$ in terms of the subspace corrections via
	\[
\eta_k := \sum _{i=1}^N \|e_i\|^2_{\mathcal V},
	\] 
and we assume the following upper and lower bounds:

	\medskip
	
	\begin{figure}[!htp]
	\begin{center}
	\begin{tikzpicture}[>=latex,scale=1.6,domain=0.2:3.2]
    \draw[thick,->] (-0.5,0.0) -- (4.5,0.0) node[right] {$v$};
    \draw[dashed,thick] (-0.5,0.5) -- (4.5,0.5) node[right] {};
    \draw[dashed,thick] (2.0,1.0) -- (4.5,1.0) node[right] {};
    \draw[dashed,thick] (3.0,2.5) -- (4.5,2.5) node[right] {};
    \draw[thick,->] (0,-0.5) -- (0,3.0) node[above] {};
    \draw[thick]   plot (\x,{0.5*(\x-1.0)^2+0.5}) node[above] {$_{E(v)}$};
    \draw[dashed,blue,thick,<->] (2.0,0.5)--(2.0,1) node[right] { };
    \draw[dashed,blue,thick,<->] (3.0,0.5)--(3.0,2.5) node[right] { };
    \draw[dashed,blue,thick,<->] (4.0,1.0)--(4.0,2.5) node[right] { };
	\draw[fill=black] (1.0,0.0) node[below left] {$u$} circle (1pt);
	\draw[fill=black] (2.0,0.0) node[below left] {$u^{k+1}$} circle (1pt);
	\draw[fill=black] (3.0,0.0) node[below left] {$u^{k}$} circle (1pt);
	\draw[fill=black] (0.0,0.5) node[below left] {$E(u)$} circle (1pt);
	\draw[fill=black] (0.0,1.0) node[below left] {$E(u^{k+1})$} circle (1pt);
	\draw[fill=black] (0.0,2.5) node[below left] {$E(u^{k})$} circle (1pt);
	\node at (2.25,0.75) {$d_{k+1}$};
	\node at (3.2,1.50) {$d_{k}$};
	\node at (4.2,1.75) {$\delta_{k}$};
	\end{tikzpicture}
	\end{center}
\caption{\LC{The sequences $\{d_k\}$ and $\{\delta_k\}$.}}	
	\label{fig:delta_k}
	\end{figure}
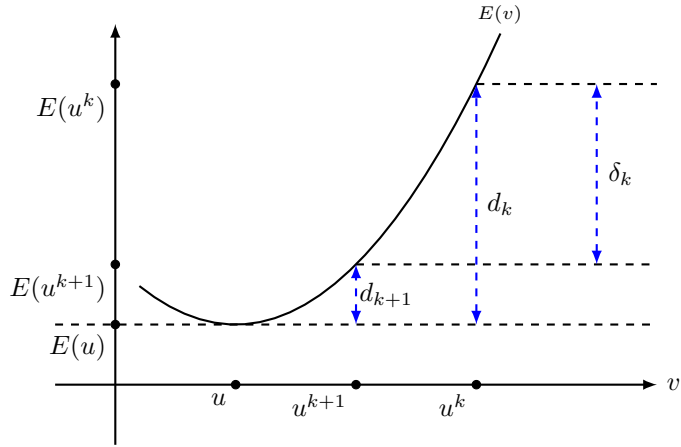

\noindent{\bf Lower Bound on Energy Decay.} There exists a positive constant $C_L$ such that for any $k=0,1,2, \cdots$
	\begin{equation}
	\label{ineq:lower-bound-corrections}	
E(u^k) - E(u^{k+1}) = \delta _k \geq  C_L\eta_k =C_L \sum _{i=1}^N \|e_i\|^2_{\mathcal V} .
	\end{equation}

\noindent{\bf Upper Bound on Optimality Gap.} There exists a positive constant $C_U$ such that for any $k=0,1,2, \cdots$
	\begin{equation}
	\label{ineq:upper-bound-corrections}	
E(u^{k+1}) - E(u) = d_{k+1} \leq C_U\eta_k = C_U\sum _{i=1}^N \|e_i\|^2_{\mathcal V} .
	\end{equation}

If these bounds hold, then, as a corollary to Theorem~\ref{thm:linear-conv}, we have

	\begin{corollary}
	\label{cor:framework}
Assume that the lower bound \eqref{ineq:lower-bound-corrections} and upper bound \eqref{ineq:upper-bound-corrections} hold with positive constants $C_L$ and $C_U$, respectively. We then have
	\[
E(u^{k+1}) - E(u)  \leq   \rho\left(E(u^k) - E(u)\right), \quad \rho: = \frac{C_U}{C_L + C_U},
	\]
and $E(u^k)$ converges monotonically, and (at least) linearly to $E(u)$, at the linear rate $\rho$. \LC{Furthermore, $u^k$ converges at least linearly to $u$.}
	\end{corollary}
	
	\begin{proof}
The linear convergence of $E(u^k)$ to $E(u)$ at the rate $\rho$ is guaranteed by Theorem~\ref{thm:linear-conv}. \LC{Using \eqref{ine:energy-bound}, with $w=u^k$, we have
	\[
\frac{\mu}{2} \nrm{u^k-u}^2_{\mathcal{V}}  \le E(u^k)-E(u) ,
	\]
which guarantees the linear convergence of $u^k$ to $u$.}
	\end{proof}

Verifying the lower bound is relatively easy since $E$ is convex. Solving the convex optimization problem in each subspace will definitely decrease the energy, and this decrease can be quantified in terms of the norms of the corrections. We make essential use of the fundamental orthogonality property.

	\begin{theorem}
	\label{thm-lower-bound-SSO}
Let $u^k$ be the $k$-th iteration and $u^{k+1} = {\rm SSO}(u^{k})$. If $E$ is strongly convex in the sense of satisfying (E1), then
	\[
\delta_k = E(u^k) - E(u^{k+1}) \geq C_L \sum _{i=1}^N \|e_i\|^2_{\mathcal V}, \quad C_L:= \frac{\mu}{2}.
	\]
	\end{theorem}
	\begin{proof}
Recalling Lemma~\ref{lm:energyatu-variant}, we observe that $J_i$ (defined in Algorithm~\ref{alg-sso}) is strictly convex over $\mathcal{V}_i$ and is Fr\'{e}chet differentiable, as it inherits the structure of $E$. It follows that
	\[
\langle J_i'(e_i),w\rangle = 0, \quad \forall \, w\in \mathcal{V}_i.
	\]
But the object on the left-hand-side is simply a directional derivative of the full energy, and it is easy to see that
	\[
\langle J_i'(e_i),w\rangle = \langle E'(v_{i-1}+e_i),w \rangle = \langle E'(v_i),w \rangle \quad  \forall \,  w \in \mathcal{V}_i.
	\]
Therefore, the fundamental orthogonality, $E'(v_i) = 0$ in $\mathcal V_i'$, holds. As $e_i = v_i-v_{i-1} \in \mathcal V_i$,  in view of Lemma \ref{lm:energyatu-variant}, we have 
	\begin{equation}
	\label{localconvexity}
E(v_{i-1}) - E(v_i) = J_i(0) - J_i(e_i) \ge \frac{\mu}{2}\nrm{e_i}_{\mathcal{V}}^2 .
	\end{equation}
The sum of the left-hand-side telescopes, and we have
	\[
E(u^k) - E(u^{k+1}) = \sum _{i=1}^N\left(  E(v_{i-1}) - E(v_i) \right)\geq \frac{\mu}{2} \sum _{i=1}^N \|e_i\|^2_{\mathcal V}.
	\]
	\end{proof}

	\begin{remark}
	\rm
In view of \eqref{localconvexity}, the convexity of $E$ can be relaxed to the local convexity of the energy sections $J_i$ in each subspace $\mathcal V_i$. Namely we may have a non-convex energy $E$ which, restricted to each subspace, is convex and the lower bound still holds. For example, the energy used in the Optimal Delaunay Triangulation (ODT)~\cite{Chen.L;Xu.J2004} is non-convex globally. But restricted to one vertex, it is convex, and the corresponding 1-D optimization problem has a closed form solution, which is known as ODT mesh smoothing~\cite{Chen.L2004}. Theorem \ref{thm-lower-bound-SSO} guarantees the energy decreasing property of ODT mesh smoothing.
	\end{remark}

The upper bound is more delicate and relies on the assumptions about the decomposition of spaces.  The result is given in the following theorem.
	\begin{theorem} 
	\label{thm:upper-bound-SSO-new}
Let $u^{k+1}$ be the $k+1^{\rm st}$ iteration in the {\rm SSO} algorithm. Suppose that the space decomposition satisfies assumptions (SS1) and (SS2) and the energy $E$ satisfies assumption (E1), then we have
	\begin{equation*}
d_{k+1} = E(u^{k+1}) - E(u) \leq C_U\sum _{i=1}^N\|e_i\|_{\mathcal V}^2, \quad  C_U := \frac{C_S^2C_A^2}{2 \mu}.
	\end{equation*}
	\end{theorem}
	\begin{proof}
Using Lemma~\ref{lm:dE-energy-V}, with the choice $v = u^{k+1}$ in~\eqref{ine:energy-dE-V}, we have
	\begin{equation*}
d_{k+1} = E(u^{k+1}) - E(u)  \leq  \frac{1}{2 \mu} \| E'(u^{k+1}) \|^2_{\mathcal{V}'}.
	\end{equation*}
For any $w \in \mathcal{V}$, we choose a stable decomposition $w = \sum_{i=1}^N w_i$, then
	\begin{align*}
\langle E'(u^{k+1}), w \rangle &= \sum_{i=1}^N \langle E'(u^{k+1}), w_i \rangle 
	\\
& = \sum_{i=1}^N \langle E'(u^{k+1}) - E'(v_i), w_i \rangle 
	\\
& = \sum_{i=1}^N\sum_{j = i+1}^{N} \langle E'(v_j) - E'(v_{j-1}), w_i \rangle 
	\\
& \leq C_S \left( \sum_{i=1}^N \| e_i \|^2_{\mathcal{V}} \right)^{1/2} \left( \sum_{i=1}^N \| w_j \|^2_{\mathcal{V}}  \right)^{1/2}  
	\\
& \leq C_SC_A \left( \sum_{i=1}^N \| e_i \|^2_{\mathcal{V}} \right)^{1/2} \| w \|_{\mathcal{V}}.
	\end{align*}
Here we use the fact that we solve the minimization problem on each subspace exactly and the energy decreases, therefore, $v_j \in \mathcal{B}$ for all $j$ and $E'(v_i) = 0$ in $\mathcal V_i'$.  Then we have
	\begin{align*}
E(u^{k+1}) - E(u) & \leq \frac{1}{2\mu}  \| E'(u^{k+1}) \|^2_{\mathcal{V}'}  
	\\
& = \frac{1}{2\mu} \left( \sup_{w \in \mathcal{V}\setminus\{0\}} \frac{\langle E'(u^{k+1}), w \rangle}{\| w \|_{\mathcal{V}}}  \right)^2
	\\
& \leq \frac{1}{2\mu} C^2_S C_A^2   \sum_{i=1}^N \| e_i \|^2_{\mathcal{V}},
	\end{align*}
which finishes the proof.
\end{proof}

Based on the lower bound given in Theorem~\ref{thm-lower-bound-SSO} and the upper bound given in Theorem~\ref{thm:upper-bound-SSO-new}, we can conclude the convergence of SSO.  Comparing with the results in \cite{Tai;Xu:2001Global}, we use slightly weaker assumptions, and the constant $C_U$ seems to be slightly better. 

	\begin{corollary}
Let $u^k$ be the $k$-th iteration and $u^{k+1} = {\rm SSO}(u^{k})$. Suppose that the space decomposition satisfies assumptions (SS1) and (SS2) and the energy $E$ satisfies assumption (E1), then we have
	\[
E(u^{k+1}) - E(u) \leq \rho ( E(u^{k}) - E(u)), \quad \mbox{with} \quad  \rho = \frac{C_S^2C_A^2}{C_S^2C_A^2 + \mu^2}.
	\]
	\end{corollary}

As we have pointed out previously, the Lipschitz continuity and constant $L$ are implicitly contained in assumption (SS2) (the Strengthened Cauchy Schwarz inequality) and the constant $C_S$.  We will show how this can be so with an application at the end of the paper. 
	
 
	\section{Fast Subspace Descent (FASD) Method: FAS with Exact Line Search} 
	\label{sec:FASD}
	
In this section, we present the theory for the convergence of the Fast Subspace Descent (FASD) method listed in Algorithm~\ref{alg:FASD}. To recap, in the SSO method, Algorithm~\ref{alg-sso}, we need to solve the optimization problem 
	\[
\min_{w\in \mathcal V_i}E(v_{i-1} + w)
	\]
in each subspace exactly, which requires evaluation of the global energy $E$ and its derivative $E'$ in the  space $\mathcal V$.  Although the size of the optimization problem is reduced to $\dim \mathcal V_i$, such evaluations are still in the original space of size $\dim \mathcal{V}$, which may be expensive.

	\subsection{Algorithm Definition}
Denote by $I_i: \mathcal V_i \hookrightarrow \mathcal V$ the natural inclusion and $R_i = I_i^{\top}: \mathcal V' \to \mathcal V_i'$ the natural restriction of functionals. Thus, for all $w\in\mathcal{V}_i$,
	\[
\langle R_i E'(v_{i-1}),w \rangle = 	\langle  E'(v_{i-1}), R_i^{\top} w \rangle = 	\langle  E'(v_{i-1}), I_i w \rangle. 
	\]
\LC{Often times we just drop $R_i$ and $I_i$, as their actions can be assumed implicitly.} We need to evaluate the gradient $R_iE'(v_{i-1} + I_i w)$, as well as \LC{the Hessian} $R_iE''(v_{i-1} + I_i w)I_i$ and its inverse, if Newton's method is used, several times. This is practical only if the natural inclusion $I_i$ is efficient to realize, e.g., a one-dimensional subspace generated by one basis function of $\mathcal V$ and the resulting method is the so-called non-linear Gauss-Seidel iteration.

\LC{Instead of solving the minimization problem using the original energy $E$, in our FASD Algorithm (Algorithm~ \ref{alg:FASD}) we utilize a locally-defined energy $E_i$ in each subspace $\mathcal V_i$ and solve a perturbed optimization problem. For the moment, let us assume that $E_i:\mathcal{V}_i\to \mathbb{R}$ is Fr\'{e}chet differentiable in $\mathcal{V}_i$. We will give further assumptions shortly.} In addition to prolongation and restriction operators, we also need a projection operator $Q_i: \mathcal V \to \mathcal V_i$. Ideally, $Q_i v$ yields a good approximation of $v$ in the subspace $\mathcal{V}_i$. Recall that as a projection operator $Q_i v_i = v_i$ for $v_i\in \mathcal V_i$.

	\begin{algorithm}
	\TitleOfAlgo{ $u^{k+1} = {\rm FASD}(u^k)$}
$v_0 = u^k$
	\;
	\For{$i = 1 : N$}{
Compute the  so-called subspace \emph{$\tau$-perturbation}: let $\xi_i = Q_iv_{i-1}$ and 
	\begin{equation}
 \tau_i  := E'_{i}(\xi_i) - R_iE'(v_{i-1}) \in\mathcal{V}_i' ;
	\end{equation}
	\\
Solve the subspace residual problem: Find $\eta_i\in \mathcal{V}_i$, such that
	\begin{equation}
	\label{eq:residual}
\langle E'_{i}( \eta_i ), w \rangle =  \langle \tau_i, w \rangle, \quad \forall \, w \in \mathcal{V}_i ;
	\end{equation}
	\\
Compute the search direction:
	\begin{equation}
s_i := \eta_i- \xi_i \in \mathcal{V}_i ;
	\end{equation}
	\\
Orthogonalize the subspace correction via the exact line search:
	\begin{equation} \label{def:line-search}
\varepsilon_i := \alpha_i^* s_i , 
	\end{equation}
where 
	\begin{equation}\label{exactline}
\alpha_i^* = \mathop{\rm argmin}_{\alpha\in\mathbb{R}} E(v_{i-1} +\alpha s_{i}); 
	\end{equation}
	\\
Apply the subspace correction:
	\begin{equation} 
v_{i} := v_{i-1} + \varepsilon_i ;
	\end{equation}
	}
$u^{k+1} := v_{N}$
	\;
	\medskip
\caption{Fast Subspace Descent (FASD) Method.}
	\label{alg:FASD}
	\end{algorithm}


We shall view our Fast Subspace Descent (FASD) method as a hybrid of the SSO and FAS methods. For the proof of convergence, it helps to treat FASD as an SSO iteration with an inexact local solver.  We could also say that FASD is essentially FAS with an additional line search step. 

Our FASD algorithm is listed in Algorithm~\ref{alg:FASD}. In the orthogonalization step, cf. \eqref{def:line-search}, we perform a line search to find the optimal step size which still requires the evaluation of some of the ``fine level" functions $E(v_{i-1} + \alpha s_i)$, $E'(v_{i-1} + \alpha s_i)$, and $E''(v_{i-1} + \alpha s_i)$ in $\mathcal V$. The computational cost is reduced compared with evaluation \LC{of} $v_{i-1} + w$ for multiple $w\in \mathcal V_i$. Algorithm~\ref{alg:FASD} is an intermediate step towards the convergences proof of original FAS. In Section~\ref{sec:FASD-approx-line}, we shall  analyze an algorithm that uses a simpler choice of step size, one that is closer to the original FAS method. \LC{In  Section~\ref{sec:Original-FAS}, we shall consider the original FAS, which corresponds to FASD with the step size $\alpha_i = 1$.} 

	\subsection{Strong Convexity of $E_i$ and Well-Posedness}
To show the well-posedness of the local problem \eqref{eq:residual}, and therefore Algorithm~\ref{alg:FASD}, we need some assumptions on the energies $E_i$. As mentioned, we assume $E_i:\mathcal{V}_i\to \mathbb{R}$ is Fr\'{e}chet differentiable for all points $v\in \mathcal V_i$. In addition, we introduce the following assumptions on the local energy, $E_i$, which is just the local version of (E1):  
	\begin{itemize}

	\item [(E3)] (Strong convexity/Ellipticity:) 
There exists a constant $\mu_i$ such that for all $v, w \in \mathcal{V}_i$
	\[
\langle E_i'(w) - E_i'(v), w - v \rangle \geq \mu_i \| w - v \|^2_{\mathcal{V}}.
	\]	
	\end{itemize}

\LC{
For the local optimization problem, expressed in equation~\eqref{eq:residual}, we are not minimizing an approximated energy $E_i$, i.e., not solving $E'_{i}(Q_iv_{i-1} + s_{i}) = 0$. Instead a so-called $\tau$-perturbation is added to the right hand side. Still, this optimization problem is uniquely solvable.
	\begin{lemma}
	\label{lm:residual}
Assume $E_i$ satisfies the strong convexity assumption (E3). Then there exists a unique solution to the residual equation \eqref{eq:residual}.
	\end{lemma}
	\begin{proof}
The residual equation \eqref{eq:residual} is the Euler equation for the minimization problem
	\begin{equation}
	\label{eq:localopt}
 \min_{v\in \mathcal V_i} \left(E_i(v) - \langle \tau_i, v\rangle\right).
	\end{equation}
As $E_i$ is strictly convex, and since the linear shift $\langle \tau_i, v\rangle$ will not affect the convexity, the global minimizer of \eqref{eq:localopt} exists, is unique, and satisfies the Euler equation~\eqref{eq:residual}. Detailed proofs can be found in~\cite[p. 35]{Ekeland;Temam:1976Convex}, \cite[Thm.~8.2-2]{Ciarlet89}, or \cite[Thm.~3.3.13]{Atkinson09}.
	\end{proof}
}
\LC{
	\begin{remark}\label{rm:FASD-SSO}\rm
We note that Algorithm~\ref{alg:FASD} (FASD) generalizes Algorithm~\ref{alg-sso} 	(SSO).  They yield the same approximations  in the case that 
	\[
E_i(\eta)  := E(v_{i-1}-Q_i v_{i-1} +\eta) , \quad \forall \ \eta \in \mathcal{V}_i.
	\]
The projection $Q_i$ just needs to satisfy the usual property $Q_i\eta = \eta$, for all $\eta\in\mathcal{V}_i$. As a consequence of this choice, $\tau_i \equiv 0$ and, for all $w\in\mathcal{V}_i$,
	\[
\langle E'(v_{i-1}+s_i), w \rangle = \langle E'(v_{i-1}-Q_iv_{i-1}+\eta_i), w \rangle = \langle E_i'(\eta_i), w \rangle = 0	.
	\]
With these choices in FASD, the last step (orthogonalization) is redundant because
	\[
\langle E'(v_{i-1}+s_i), s_i \rangle = 0
	\]
upon taking $w = s_i$. In other words, the orthogonality is valid with $\alpha_i^* = 1$ for SSO.
	\end{remark}
}

	\subsection{Lower bound}
\LC{The first correction that we obtain in Algorithm~\ref{alg:FASD}, namely, $s_i = \eta_i -\xi_i$, where $\xi_i=Q_iv_{i-1}$ is the full approximation, is used as the search direction for a line optimization. The line optimization confers an orthogonalization property to the corrected approximation $v_i$. Due to this orthogonalization and the convexity of $E$, the proof of the lower bound for FASD is almost exactly the same as that for the SSO method.} 

	\begin{theorem}
	\label{thm-lower-bound-FASD}
Suppose that $E$ satisfies (E1) and $E_i$ satisfies (E3), and let $u^k$ be the $k^{\rm th}$ iteration in the FASD algorithm (Algorithm~\ref{alg:FASD}). Then
	\[
 E(u^k) - E(u^{k+1}) \geq \frac{\mu}{2} \sum _{i=1}^N \|\varepsilon_i\|^2_{\mathcal V}.
	\]
	\end{theorem}
	\begin{proof}
We apply a similar technique as in the proof of Theorem~\ref{thm-lower-bound-SSO}. Due to the line search, we still have an orthogonality property that can be utilized, namely,
	\[
\langle E'(v_{i}), w \rangle = 0, \quad w \in {\rm span} \{ s_i\}.
	\]
Then, applying Lemma \ref{lm:energyatu-variant}, with the subspace  $\mathcal{W} = {\rm span} \{ s_i\}$, and noting that 
	\[
v_i - v_{i-1} = \varepsilon_i = \alpha_i^* s_i \in {\rm span}\{ s_i\},
	\]
we have 
	\[
E(v_{i-1}) - E(v_i) \geq \frac{\mu}{2}\|v_{i-1}-v_i\|^2_{\mathcal V} = \frac{\mu}{2}\|\varepsilon_i\|^2_{\mathcal V},
	\]
and consequently
	\[
E(u^k) - E(u^{k+1}) = \sum _{i=1}^N\left( E(v_{i-1}) - E(v_i)\right) \geq \frac{\mu}{2} \sum _{i=1}^N \|\varepsilon_i\|^2_{\mathcal V}.
	\]
	\end{proof}
	
We will later need the following simple result, which follows because of the strong convexity assumption (E3).
	\begin{lemma}
	\label{lm:si}
Let $s_i$ be computed as in Algorithm~\ref{alg:FASD} and suppose that $E_i$ satisfies assumption (E3). Then $s_i$ is a descent direction in the sense that
 	\[
 \langle - E'(v_{i-1}), s_i \rangle \geq \mu_i \|s_i\|_{\mathcal V}^2 > 0.
 	\]
	\end{lemma}
\begin{proof}
The local problem \eqref{eq:residual} can be rewritten as follows: find $\eta_i\in \mathcal V_i$ s.t. 
	\begin{equation}
	\label{localproblem}
\langle E_{i}'(\eta_i) - E_{i}'(\xi_i), w \rangle = -\langle R_iE'(v_{i-1}), w \rangle, \quad \forall \ w \in \mathcal V_i.
	\end{equation}
Here recall that $\xi_i = Q_iv_{i-1} \in \mathcal{V}_i$ and $\eta_i = \xi_i + s_i$. Choosing $w = s_i$ and using the strong convexity of $E_i$, we obtain the inequality
	\[
 \langle -R_iE'(v_{i-1}), s_i \rangle = \langle E_{i}'(\eta_i) - E_{i}'(\xi_i), s_i \rangle \geq \mu_i \|s_i\|_{\mathcal V}^2 > 0.
	\]
	\end{proof}
	
	\subsection{Lipschitz continuity of $E_i'$ and estimates of $\alpha_i^*$}
As Theorem \ref{thm-lower-bound-FASD} implies, the energy is always decreasing and iterates will remain in the \LC{sublevel set} $\mathcal{B}$, but the search region, and, e.g., the point $\xi_i + s_i$, may not be contained in $\mathcal B$. To be able to use Lipschitz continuity, we introduce an enlarged set
	\begin{equation}
	\label{Bbar}
{\mathcal B}^+ := \left\{v \in  \mathcal V \ \middle| \  \operatorname{dist}(v, \mathcal B) \leq \sqrt{\chi} \right\},
	\end{equation}
where $\chi$ is given by
	\[
\chi := \frac{2L^2}{\mu \min_i\mu_i^2 } (E(u_0) - E(u)).
	\]
We then introduce a  Lipschitz continuity of $E_i'$ with respect to \LC{the projection of ${\mathcal B}^+$}:
	\begin{itemize}
 	\item[(E4)]
 (Lipschitz continuity of the first order derivative:) There exists a constant $L_i>0$, such that
	\begin{equation*}
\| E_i'(w) - E_i'(v) \|_{\mathcal{V}'} \leq L_i \| w - v \|_{\mathcal{V}},
	\end{equation*}
for all $w, v \in \mathcal{B}_i : = Q_i \mathcal{B}^+$.
	\end{itemize}
	
	\LC{\begin{remark}\rm
Observe that we must assume that (E3) holds for (E4) to make sense. In other words, we cannot assume (E4) without first assuming (E3), since $\mu_i$ is involved in the definition of $\chi$ and, therefore, ${\mathcal B}^+$. Regarding ${\mathcal B}^+$, note that it is not a sublevel set. However, it is straightforward to verify that both $\mathcal{B}^+$ and  $\mathcal{B}_i = Q_i \mathcal{B}^+$ are convex. The proofs are omitted for the sake of brevity. 	
	\end{remark}}

Later, we will show that $\xi_i + s_i \in \mathcal{B}_i$ so that we can take advantage of the Lipschitz continuity of $E_i'$ in our analysis. Notice that the Lipschitz continuity of $E_i'$ is imposed for the set $ Q_i{\mathcal B}^+$, which is related to $\mathcal B$ used in (E2). Interestingly, there is no relationship between $E$ and $E_i$ that is explicitly assumed for the moment. Indeed $E$ and $E_i$ are just related through the upper and lower bound of the first derivatives and norms. In general, based on the assumptions (E3) and (E4), we have the following lemma, which gives results analogous to those in Lemma~\ref{lm:dE} and~\ref{lm:energyatu-V}. 

	\begin{lemma}
	\label{lm:Ei-V}
Assume $E_i$ satisfies assumptions (E3) and (E4). For any $v, w \in \mathcal{B}_i$, 
\[
\mu_i \nrm{w-v}^2_{\mathcal{V}} \le \langle E_i'(w) -E_i'(v) , w-v\rangle  \le L_i\nrm{w-v}^2_{\mathcal{V}},
\]
and 
	\[
\frac{\mu_i}{2} \nrm{w-v}^2_{\mathcal{V}}  + \langle E_i'(v), w-v \rangle \le E_i(w)-E_i(v) \le  \langle E_i'(v), w-v \rangle + \frac{L_i}{2}\nrm{w-v}_{\mathcal{V}}^2.
	\]
Though it is not required, if it happens that $u_i \in \mathcal{B}_i$, where $u_i\in{\mathcal{V}}_i$ is the global minimizer of $E_i$, then for all $w \in \mathcal{B}_i$, 
	\[
\frac{\mu_i}{2} \nrm{w-u_i}^2_{\mathcal{V}}  \le E_i(w)-E_i(u_i) \le   \frac{L_i}{2}\nrm{w-u_i}_{\mathcal{V}}^2.
	\]
The lower bounds above hold for all $w\in \mathcal V_i$, without restriction.
	\end{lemma}

In order to better understand the choice of the step size, we introduce the scalar function $f_i$. See Figure~\ref{fig:fplot} and Equation~\eqref{eqn:energy-section-1d}.

	\begin{figure}[!htp]
	\begin{center}
	\begin{tikzpicture}[>=latex,scale=1.6,domain=-0.2:4.2]
    \draw[thick,->] (-0.5,0.0) -- (4.5,0.0) node[right] {$\alpha$};
    \draw[dashed,thick] (-0.5,0.5) -- (4.5,0.5) node[right] {};
    \draw[thick,->] (0,-1.5) -- (0,1.2) node[above] {};
    \draw[thick]   plot (\x,{0.0625*(\x-2.0)^4 + 0.25*(\x-2.0)^2 -1.5}) node[above] {$f_i(\alpha)$};
    \draw[fill=black] (4.0,0) node[below right] {$\alpha_{L,i}^o$} circle (1pt);
    \draw[fill=black] (4.0,0.5) node[below right] {$f_i(\alpha_{L,i}^o)$} circle (1pt);
    \draw[dashed,blue,thick] (2.0,-1.5)--(2.0,1);
     \draw[fill=black] (2.0,0) node[above right] {$\alpha_i^*$} circle (1pt);
     \draw[fill=black] (0.0,0.5) node[below left] {$f_i(0)$} circle (1pt);
     \draw[fill=black] (2.0,-1.5) node[below left] {$f_i(\alpha_i^*)$} circle (1pt);
     \draw[fill=black] (3.0,0.0) node[below left] {$\alpha_{L,i}$} circle (1pt);
	\end{tikzpicture}
	\end{center}
\caption{\LC{The function $f_i$ defined in \eqref{eqn:energy-section-1d}. $f_i$ is a one-dimensional energy section. It is straightforward to prove that its minimizer, $\alpha_i^*$ is positive.}}	
	\label{fig:fplot}
	\end{figure}
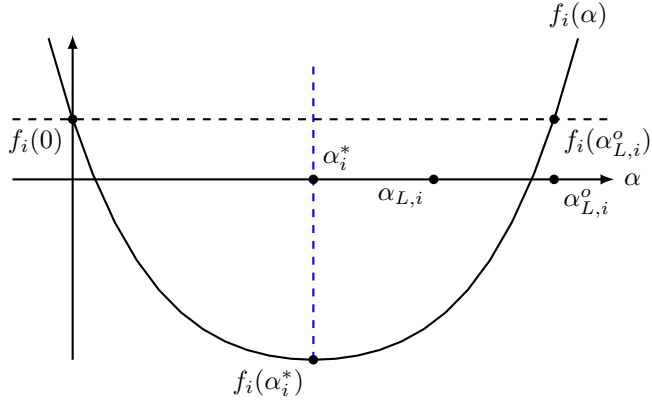

\LC{
	\begin{proposition}
Suppose that $E$ satisfies assumption (E1) and the local energy $E_i$ satisfies assumption (E3). Define the one-dimensional energy section
	\begin{equation}
	\label{eqn:energy-section-1d}
f_i(\alpha) := E(v_{i-1}+ \alpha s_i).
	\end{equation}
Then
	\[
f_i'(0) = \langle E'(v_{i-1}), s_i \rangle \le - \mu_i\nrm{s_i}_{\mathcal{V}}^2.
	\]
Furthermore, $\alpha_i^* > 0$, and, for all $\alpha \in (0, \alpha_i^*]$, $f_i(\alpha) < f_i(0)$.
	\end{proposition}
	\begin{proof}
Lemma \ref{lm:si} implies $f_i'(0) < 0$. As $f_i'$ is continuous, we conclude that the minimizing point is positive, $\alpha_i^* > 0$, and for all $\alpha \in (0, \alpha_i^*]$, $f_i(\alpha) < f_i(0) = E(v_{i-1})$.
	\end{proof}
}

	\begin{lemma}
	\label{lm:falpha}
Assume $E$ satisfies assumptions (E1) -- (E2). Then $f_i(\alpha)$, defined in \eqref{eqn:energy-section-1d}, is differentiable and strongly convex in the following sense: for all $\alpha, \beta \in \mathbb R$,
	\[
(f_i'(\alpha) - f_i'(\beta))(\alpha - \beta) \geq (\alpha - \beta)^2\mu \|s_i\|_{\mathcal V}^2 .
	\]
Furthermore, $f_i'$ is Lipschitz in the following sense: for all $0\le \alpha, \beta \leq \alpha_{L,i}$, 
	\[
|f_i'(\alpha) - f_i'(\beta)|  \leq L\|s_i\|_{\mathcal V}^2 |\alpha - \beta | ,
	\]
where $\alpha_{L,i} := (1+ \sqrt{\mu/L})\alpha_i^*$.
	\end{lemma}
	\begin{proof}
The proof is based on the following identity
	\[
f_i'(\alpha) - f_i'(\beta) =   \langle E'(v_{i-1}+ \alpha s_i) - E'(v_{i-1}+ \beta s_i) , s_i \rangle.
	\]
Then, by assumption (E1),
	\begin{align*}
(f_i'(\alpha) - f_i'(\beta))(\alpha-\beta) = & \ \langle E'(v_{i-1}+ \alpha s_i) - E'(v_{i-1}+ \beta s_i) , \alpha s_i-\beta s_i \rangle	
	\\
\ge & \ \mu \|(\alpha-\beta) s_i\|^2_{\mathcal{V}}.
	\end{align*}

To use the Lipschitz inequality involving $E'$, we need to ensure that the points of evaluation are inside the set $\mathcal B$, which imposes an upper bound on $\alpha$ and $\beta$. As $f_i'(0)<0$ and $f_i'(\alpha_i^*) = 0$,  by coercivity, there exists $\alpha_{L,i}^o > \alpha_i^*$,   such that $f_i(0) = f_i(\alpha_{L,i}^o)$, and, for all $\alpha \in (0, \alpha_{L,i}^o)$, $f_i(\alpha) < f_i(0)$. This implies $v_{i-1}+ \alpha s_i \in \mathcal B$, for all $\alpha \in (0, \alpha_{L,i}^o).$ 

We then show $f_i'$ is Lipschitz with constant $L\|s_i\|_{\mathcal V}^2$ on the interval $[0,\alpha_{L,i}^o]$. 
%
For all $\alpha,\beta\in (0,\alpha_{L,i}^o)$, $\alpha\ne\beta$, 
	\begin{align*}
|f_i'(\alpha) - f_i'(\beta)| & = |\langle E'(v_{i-1}+\alpha s_i) - E'(v_{i-1}+\beta s_i) ,s_i \rangle| 
	\\
& = \frac{1}{|\alpha-\beta|} |\langle E'(v_{i-1}+\alpha s_i) - E'(v_{i-1}+\beta s_i) , (\alpha-\beta)s_i \rangle|
	\\
& \le  \frac{1}{|\alpha-\beta|} L \|(\alpha-\beta) s_i\|^2_{\mathcal{V}}
	\\
& =  L \|s_i\|^2_{\mathcal{V}} |\alpha-\beta| .
	\end{align*}
	
We now estimate $\alpha_{L,i}^o$. As $f_i'(\alpha_i^*) = 0$ and $f_i'$ is Lipschitz in $(0, \alpha_{L,i}^o)$, we have 
%
	\begin{align*}
0< f_i(\alpha_{L,i}^o) - f_i(\alpha^*_i) \leq
	 (\alpha_{L,i}^o - \alpha_i^*)^2\frac{L}{2} \|s_i\|_{\mathcal V}^2.
	\end{align*}
On the other hand, and again from Lemma~\ref{lm:energyatu-variant},
	\[
f_i(\alpha_{L,i}^o) - f_i(\alpha_i^*) = f_i(0) - f_i(\alpha_i^*) \geq \frac{\mu (\alpha_i^*)^2}{2} \|s_i\|_{\mathcal V}^2.
	\]
The desired bound 
	\[
\alpha_{L,i}^o \geq \alpha_{L,i}:=\left(1+ \sqrt{\frac{\mu}{L}}\right)\alpha_i^* >\alpha_i^*>0
	\]
then follows. 
Note that we need only $f_i'$ is Lipschitz with the same constant on the smaller interval $[0,\alpha_{L,i}] \subseteq [0,\alpha_{L,i}^o]$. The proof is complete.
	\end{proof}

To use the Lipschitz continuity of $E_i$, we require $\xi_i + s_i\in \mathcal B_i = Q_i \mathcal{B}^+$, which will be proved by a lower bound of the optimal step size.
 
	\begin{lemma}
	\label{lm:alphastarlower}
Assume the energy $E$ satisfies  assumptions (E1) -- (E2) and the local energy $E_i$ satisfies the strong convexity assumption (E3), then we have the lower bound 
 	\[
\frac{\mu_i}{L} \leq \alpha_i^*.
 	\]
Consequently,
 	\[
\alpha_{L,i}^o \geq \alpha_{L,i}:=\left(1+ \sqrt{\frac{\mu}{L}}\right)\alpha_i^* >\alpha_i^* \ge  \frac{\mu_i}{L} > 0.
 	\]
	\end{lemma}
	\begin{proof}
Recall that $\varepsilon_i = \alpha_i^* s_i\in {\rm span} \{ s_i\}$, and, due to the line search, we still have an orthogonality property that can be utilized, namely,
	\[
\langle E'(v_{i}), w \rangle = 0, \quad \forall \, w \in {\rm span} \{ s_i\}.
	\]
Thus $E'(v_{i-1} + \varepsilon_i) =0 $ in the dual of ${\rm span} \{ s_i\}$. By step 2 in the FASD Algorithm~\ref{alg:FASD},
 	\[
- E'(v_{i-1}) = E_i'(\xi_i + s_i) - E_i'(\xi_i) \quad \mbox{in} \quad \mathcal V_i'.
	\] 
The lower bound is obtained by the strong convexity of $E_i$ and Lipschitz continuity of $E'$:
	\begin{align*}
\alpha_i^* L \|s_i \|_{\mathcal V}^2 = \frac{1}{\alpha_i^*} L \|\varepsilon_i \|_{\mathcal V}^2 & \geq  \frac{1}{\alpha_i^*}\langle E'(v_{i-1} + \varepsilon_i) - E'(v_{i-1}), \varepsilon_i \rangle 
	\\
& = \langle E'(v_{i-1} + \varepsilon_i) - E'(v_{i-1}), s_i \rangle 
	\\
& =  -\langle E'(v_{i-1}), s_i \rangle 
	\\
& =  \langle E_i'(\xi_i + s_i) - E_i'(\xi_i), s_i \rangle
	\\
& \geq   \mu_i \| s_i \|_{\mathcal V}^2.
	\end{align*}
Note that $v_{i-1} + \varepsilon_i\in \mathcal B$ by Lemma \ref{lm:falpha} so that we can use Lipschitz continuity of $E'$.
	\end{proof}

Next we show the norm of $s_i$ is bounded and thus $\xi_i + s_i\in \mathcal B_i$.
	\begin{lemma}
	\label{lm:Bi}
The point $\xi_i + s_i$ is in the set $\mathcal B_i$.
	\end{lemma}
	\begin{proof}
To show that $\xi_i+s_i\in\mathcal{B}_i$, it suffices to show that $v_{i-1}+s_i \in {\mathcal B}^+$, since $\xi_i+s_i = Q_i(v_{i-1}+s_i)$. To start, we know that $v_{i-1}\in \mathcal B$; so by the definition of  ${\mathcal B}^+$ in~\eqref{Bbar}, it suffices to prove that $\| s_i \|_\mathcal{V}^2\le \chi$. By Theorem \ref{thm-lower-bound-FASD} and Lemma \ref{lm:alphastarlower}, we have
	\[
\frac{\mu_i^2}{L^2} \| s_i\|_{\mathcal V}^2 \leq (\alpha_i^*)^2 \| s_i\|_{\mathcal V}^2 =  \| \varepsilon_i \|_{\mathcal V}^2 \leq \frac{2}{\mu} (E(v_{i-1}) - E(v_i))\leq  \frac{2}{\mu} (E(u_0) - E(u)),
	\]
which implies
	\[
 \| s_i\|_{\mathcal V}^2 \le   \frac{2L^2}{\mu\min_i \mu_i^2 } (E(u_0) - E(u)) = \chi.
	\]
Therefore,
	\[
\operatorname{dist}(\mathcal{B},v_{i-1}+s_i) \le \|s_i\|_\mathcal{V}\le \sqrt{\chi},
	\]
and the result is proven.
	\end{proof}

\subsection{Upper bound}

With our estimates of $\alpha_i^*$ in place, we are now ready to establish an upper bound for the iterates in our FASD Algorithm~\ref{alg:FASD}.
	\begin{theorem}
	\label{thm:upper-FASD}
Suppose the space decomposition satisfies (SS1) and (SS2), the energy $E$ satisfies (E1) -- (E2), and $E_i$ satisfies (E3) -- (E4). Then we have the upper bound
	\[
E(u^{k+1}) - E(u) \leq C_U\sum _{i=1}^N\| \varepsilon_i \|_{\mathcal V}^2,
	\]
	where $ C_U:= C_A^2 \left[  C_S+ L  \left( 1 + \max_i \{  L_i/\mu_i \} \right)  \right]^2/(2 \mu).$
	\end{theorem}
	\begin{proof}
Note, for any $w \in \mathcal{V}$, we choose a stable decomposition $w = \sum_{i=1}^N w_i$, then
	\begin{align*}
\langle E'(u^{k+1}), w \rangle &= \sum_{i=1}^N \langle E'(u^{k+1}), w_i \rangle 
	\\
& = \sum_{i=1}^N \left(\langle E'(u^{k+1}) - E'(v_i), w_i \rangle + \langle E'(v_i), w_i \rangle \right)
	\\
& = {\rm I}_1 + {\rm I}_2,
	\end{align*}
where
	\begin{equation*}
{\rm I}_1 : = \sum_{i=1}^N \langle E'(u^{k+1}) - E'(v_i), w_i \rangle \quad \mbox{and} \quad {\rm I}_2 := \sum_{i=1}^N  \langle E'(v_i), w_i \rangle .
	\end{equation*}
Using the stability of the decomposition (SS1) and the strengthened Cauchy-Schwartz inequality (SS2), ${\rm I}_1$ can be estimated in exactly the same way as in Theorem~\ref{thm:upper-bound-SSO-new}. Therefore,
	\begin{equation*}
{\rm I}_1 \leq C_S C_A \left (\sum_{i=1}^N \|\varepsilon_i\|_{\mathcal V}^2\right )^{1/2}\|w\|_{\mathcal V}.
	\end{equation*}
For ${\rm I}_2$, we insert $\tau_i-E_i'(\xi_i + s_i)$, which is zero in $\mathcal{V}_i'$,  use the Lipschitz continuities, the standard Cauchy-Schwartz inequality, to get
	\begin{align*}
{\rm I}_2 & =  \sum _{i=1}^N \langle E'(v_i) - E'(v_{i-1}) - E_i'(\xi_i + s_i) + E_i'(\xi_i), w_i \rangle
	\\
& \leq \sum_{i=1}^N \left(  L \| \varepsilon_i \|_{\mathcal{V}} + L_i \| s_i \|_{\mathcal{V}} \right) \| w_i \|_{\mathcal{V}}
	\\
& \le  L \sum_{i=1}^N \left(   1 + \frac{L_i}{\mu_i}  \right)\| \varepsilon_i \|_{\mathcal{V}} \| w_i \|_{\mathcal{V}}
	\\
& \leq L C_A  \left( 1 + \max_{1\le i\le N} \frac{ L_i}{\mu_i}  \right)\left (\sum_{i=1}^N \|\varepsilon_i\|_{\mathcal V}^2\right )^{1/2}\| w\|_{\mathcal V} .
		\end{align*}
In the last estimate, we used the relation $s_i = {\alpha_i^*}^{-1}\varepsilon_i$ and the lower bound of $\alpha_i^*$ given in Lemma~\ref{lm:alphastarlower}. 

Putting the estimates together, we have, 
	\[
\nrm{E'(u^{k+1})}_{\mathcal{V}'}^2 \le C_A^2 \left[ C_S + L \left( 1 + \max_{1\le i\le N} \frac{L_i}{\mu_i} \right)\right]^2 \sum_{i=1}^N \|\varepsilon_i\|_{\mathcal V}^2 .
	\]
Using inequality~\eqref{ine:energy-dE-V} in Lemma~\ref{lm:dE-energy-V} with $v = u^{k+1}$,  the result follows.
	\end{proof}

	\begin{remark}\rm 
	\label{rem:quadratic-Ei}
Our theory suggests we can simply choose 
	\begin{equation}
	\label{linearEi}
E_i(w) = \frac{1}{2}\|w - \xi_i\|_{\mathcal V}^2 = \frac{1}{2}\|w - Q_i v_{i-1}\|_{\mathcal V}^2, \quad \forall \, w\in \mathcal{V}_i;
	\end{equation}
for then (E3) and (E4) hold with $L_i = \mu_i = 1$. Moreover, the local problem becomes like that of the linear preconditioned gradient descent method:
	\begin{equation}
\left(  \eta_i -\xi_i , w \right)_{\mathcal{V}_i} =  - \langle  R_iE'(v_{i-1}), w \rangle, \quad \forall \, w \in \mathcal{V}_i .
	\end{equation}
In this case \eqref{localproblem} has the closed form solution
	\[
\eta_i - \xi_i =: s_i = - \mathfrak{R}_i R_i E'(v_{i-1}),
	\]
where $\mathfrak{R}_i$ is the Riesz map $\mathcal V_i' \to \mathcal V_i$ and its realization is the inverse of an SPD matrix of size $\dim \mathcal V_i$.   In fact, we can even use, for any fixed $g_i\in \mathcal{V}_i$ that we like,
	\begin{equation*}
E_i(w) = \frac{1}{2}\|w - g_i\|_{\mathcal V}^2, \quad \forall \, w\in \mathcal{V}_i,
	\end{equation*}
and the same basic result is true (by linearity):  $s_i = - \mathfrak{R}_i R_i E'(v_{i-1})$. 	
\end{remark}

In any case, solving a linear local problem can dramatically reduce the computational cost of the FASD. See Section~\ref{sec:numerics} for a practical discussion of this point. In this setting, FASD is closely related to the coordinate descent methods analyzed in~\cite{Nesterov:2012Efficiency}. See also, for example,~\cite{Feng;Salgado;Wang;Wise:2017Preconditioned}. \LC{Another advantage of using \eqref{linearEi} is that one does not need to worry about the particular choice of ${\mathcal B}_i$. The quadratic energy in \eqref{linearEi} is globally Lipschitz.}

	\begin{remark} \rm
 We can also choose the local quadratic energy
 	\begin{equation}
	\label{NewtonEi}
E_i(w) = \frac{1}{2}\|w - \xi_i\|_{E''(\xi_i)}^2 : = \frac{1}{2} \langle E''(\xi_i)( w - \xi_i), w - \xi_i \rangle, \quad \forall w\in \mathcal V_i.
	\end{equation}
\LC{Here, recall that $\xi_i = Q_iv_{i-1}$ and $E''(\xi_i)$ should be understood as the restriction of the bilinear form $E''(\xi_i)$ on subspace $\mathcal{V}_i\times \mathcal V_i$.} 
 Then the local problem becomes one \LC{damped} Newton's iteration in subspace $\mathcal V_i$
	\[
s_i = - (R_iE''(\xi_i)I_i)^{-1}R_i E'(v_{i-1}).
	\]
In this setting, the block Newton's method proposed in~\cite{Lu:2017Randomized} can be interpreted as a FASD with space decomposition. We will investigate the randomized version in a future paper.
	\end{remark}
%

	\begin{corollary}
In addition to the hypotheses of the last theorem, let us assume that $E_i$ is quadratic, chosen as in \eqref{linearEi}. Then, 
	\[
E(u^{k+1}) - E(u) \leq \frac{C_A^2(C_S+ 2L)^2 }{2 \mu} \sum _{i=1}^N\| \varepsilon_i \|_{\mathcal V}^2 .
	\]
	\end{corollary}

\subsection{Convergence}
Using Theorems \ref{thm-lower-bound-FASD} and \ref{thm:upper-FASD}, and Corollary~\ref{cor:framework}, we obtain the following linear convergence result.
	\begin{corollary}
Let $u^k$ be the $k$-th iteration and $u^{k+1} = {\rm FASD}(u^{k})$. Suppose that the space decomposition satisfies assumptions (SS1) and (SS2), the energy $E$ satisfies assumption (E1) -- (E2), and the energy $E_i$ satisfies assumption (E3) -- (E4), then we have
	\[
E(u^{k+1}) - E(u) \leq \rho ( E(u^{k}) - E(u)),
	\]
with
	\[
\rho = \frac{C_A^2 \left[  C_S+ L  \left( 1 + \max_i \{  L_i/\mu_i \} \right)  \right]^2}{C_A^2 \left[  C_S+ L  \left( 1 + \max_i \{  L_i/\mu_i \} \right)  \right]^2 + \mu^2}.
	\]
Furthermore if $E_i$ is quadratic, chosen as in \eqref{linearEi}, then
	\[
\rho = \frac{C_A^2 \left( C_S+ 2L\right)^2}{C_A^2 \left( C_S+ 2L\right)^2 + \mu^2}.
	\]
	\end{corollary}


	\section{FASD with Approximate Line Search} 
	\label{sec:FASD-approx-line}

In this section, we consider the FASD algorithm with approximated line search. The method is detailed in Algorithm~\ref{alg:FASD-approx}. The key difference between this algorithm and Algorithm~\ref{alg:FASD} is that a fixed step size $\alpha_i$ is employed rather than computing $\alpha_i^*$ via a line search. In this case, there is no need to repeatedly evaluate $E$ and its derivatives in the subspace. We need only compute $R_iE'(v_{i-1})$ once for the local problem (for use in the computation of $\tau_i$ and $\alpha_i$). 

In the next section, Section~\ref{sec:Original-FAS}, we shall also consider the original FAS, which corresponds to FASD with the step size $\alpha_i = 1$. We prove its convergence based on an additional approximation property.

	\begin{algorithm}[htp!]
\TitleOfAlgo{ $u^{k+1} = {\rm FASD-ALS}(u^k)$}
$v_0 = u^k$
	\;
\For{$i = 1 : N$}{
Compute the subspace \emph{$\tau$-perturbation}: let $\xi_i = Q_iv_{i-1}$ and 
	\begin{equation}
 \tau_i  := E'_{i}(\xi_i) - R_iE'(v_{i-1}) \in\mathcal{V}_i' ;
	\end{equation}
	\\
Solve the subspace residual problem: Find $\eta_i\in \mathcal{V}_i$, such that
	\begin{equation}
	\label{eq:residual-1}
\langle E'_{i}( \eta_i ), w \rangle =  \langle \tau_i, w \rangle, \; \forall w \in \mathcal{V}_i.	
	\end{equation}
	\\
Compute the search direction and the quadratic step size:
	\begin{align}
s_i &:= \eta_i- \xi_i \in \mathcal{V}_i , 
	\\
\LC{\alpha_i^q} &:= -\frac{\langle R_iE'(v_{i-1}), s_i \rangle}{L \|s_i\|_{\mathcal V}^2}.
	\label{eq:alphai} 
	\end{align}
Apply the subspace correction:
	\begin{equation}
v_{i} := v_{i-1} + \alpha_i^q s_i.
	\end{equation}
}
$u^{k+1} := v_{N}$
	\;
	\medskip
\caption{FASD algorithm with approximate line search (ALS).} 
	\label{alg:FASD-approx}
	\end{algorithm}

Recall the scalar function $f_i(\alpha) := E(v_{i-1}+ \alpha s_i)$, with $f_i(0) = E(v_{i-1})$, $f_i'(0) = \langle E'(v_{i-1}) , s_i \rangle < 0$. Using $f_i(0)$ and $f_i'(0)$, we define the quadratic function 
	\begin{equation}
q_i(\alpha) := f_i(0) + f_i'(0)\alpha + \frac{L\|s_i\|_{\mathcal{V}}^2}{2}\alpha^2.
	\label{eqn:q_i}
	\end{equation}
The optimal step size for FASD is, of course, $\alpha_i^* = \mathop{\rm argmin}_{\alpha\in\mathbb{R}}f_i(\alpha)$. Our choice for this new algorithm is 
	\[
\alpha_i^q = \mathop{\rm argmin}_{\alpha\in\mathbb{R}}q_i(\alpha) =  -\frac{f_i'(0)}{L \| s_i\|_{\mathcal V}^2} =  -\frac{\langle R_iE'(v_{i-1}), s_i \rangle}{L \| s_i\|_{\mathcal V}^2} ,
	\]
which satisfies the following estimate:
	\begin{lemma}
	\label{lm:alphai}
Assume the energy $E$ satisfies the Lipschitz continuity assumption (E2) and the local energy $E_i$ satisfies the strong convexity assumptions (E3), then
	\[
\frac{\mu_i}{L} \leq \alpha_i^q \leq \alpha_i^*.
	\]
	\end{lemma}
	\begin{proof}
The lower bound is obtained by the definition of $\alpha_i^q$ and Lemma \ref{lm:si}. To prove the upper bound, we notice that, due to line search,
	\[
f_i'(\alpha_i^*) =  \langle E'(v_{i-1} + \alpha_i^*s_i),s_i\rangle = 0
	\]
and thus
	\[
\alpha_i^q L\| s_i\|_{\mathcal V}^2  = -\langle R_iE'(v_{i-1}), s_i \rangle = \langle E'(v_{i-1} + \alpha_i^*s_i) - E'(v_{i-1}), s_i \rangle\leq \alpha_i^* L\|s_i\|_{\mathcal V}^2.
 	\]
	\end{proof}
	
	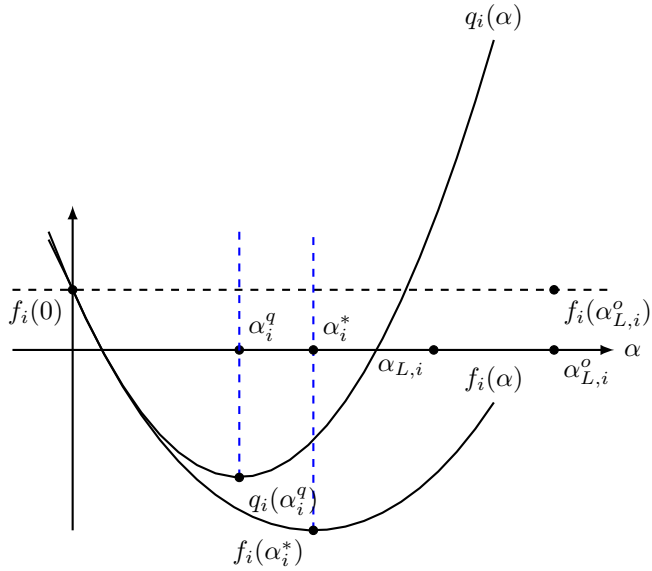
\begin{figure}[!htp]
	\begin{center}
	\begin{tikzpicture}[>=latex,scale=1.6,domain=-0.2:3.5]
    \draw[thick] plot (\x,{(1/64)*(\x-2.0)^4+(7/16)*(\x-2.0)^2-1.5}) node[above] {$f_i(\alpha)$};
    \draw[thick] plot (\x,{0.5 -(9/4)*\x +(13/16)*\x^2}) node[above] {$q_i(\alpha)$};
    \draw[thick,->] (-0.5,0.0) -- (4.5,0.0) node[right] {$\alpha$};
    \draw[dashed,thick] (-0.5,0.5) -- (4.5,0.5) node[right] {};
    \draw[thick,->] (0,-1.5) -- (0,1.2) node[above] {};
    \draw[fill=black] (4.0,0) node[below right] {$\alpha_{L,i}^o$} circle (1pt);
    \draw[fill=black] (4.0,0.5) node[below right] {$f_i(\alpha_{L,i}^o)$} circle (1pt);
    \draw[dashed,blue,thick] (2.0,-1.5)--(2.0,1);
     \draw[fill=black] (2.0,0) node[above right] {$\alpha_i^*$} circle (1pt);
     \draw[fill=black] (18/13,0) node[above right] {$\alpha_i^q$} circle (1pt);
    \draw[dashed,blue,thick] (18/13,-1.06)--(18/13,1);
     \draw[fill=black] (0.0,0.5) node[below left] {$f_i(0)$} circle (1pt);
     \draw[fill=black] (2.0,-1.5) node[below left] {$f_i(\alpha_i^*)$} circle (1pt);
     \draw[fill=black] (18/13,-1.06) node[below right] {$q_i(\alpha_i^q)$} circle (1pt);
     \draw[fill=black] (3.0,0.0) node[below left] {$\alpha_{L,i}$} circle (1pt);
	\end{tikzpicture}
	\end{center}
\caption{\LC{The functions $f_i$, defined in \eqref{eqn:energy-section-1d}, and its quadratic approximation $q_i$, defined in \eqref{eqn:q_i}. The quadratic minimizer $\alpha_i^q$ is always to the left of $\alpha_i^*$ by construction.}}	
	\label{fig:f-q_plot}
	\end{figure}

%

Now, since the optimal linear search procedure is broken, the orthogonality conditions with respect to the corrections are broken, and establishing the lower bound is a little more complicated.

	\begin{theorem}
Let $u^k$ be the $k$-th iteration and $u^{k+1} = {\rm FASD-ALS}(u^{k})$. Suppose that $E$ satisfies assumption (E1) -- (E2) and the local energy $E_i$ is strongly convex, satisfying assumption (E3). Then, we have
	\begin{equation*}
E(u^{k}) - E(u^{k+1}) \geq C_L \sum _{i=1}^N \| \alpha_i^q s_i\|^2_{\mathcal V}, \quad C_L = \frac{L}{2}.
	\end{equation*}  
	\end{theorem}
	\begin{proof}
It suffices to prove 
	\[
E(v_{i-1}) - E(v_i) = f(0) - f(\alpha_i^q) \geq \frac{L}{2}\|\alpha_i^q s_i\|^2.
	\]
By Lemma \ref{lm:falpha}, for $\alpha \in [0, \alpha_{L,i}]$, $f_i'$ is Lipschitz continuous with constant $L\|s_i\|^2_{\mathcal V}$. Then for $\alpha \in [0, \alpha_{L,i}]$,
\begin{align*}
f_i(\alpha) - q_i(\alpha) & = f_i(\alpha) -  f_i(0) - \alpha f_i'(0) - \frac{L\|s_i\|^2_{\mathcal V}}{2} \alpha^2\leq 0.
\end{align*}
Namely $f_i(\alpha) \leq q_i(\alpha)$ for all $\alpha \in [0, \alpha_{L,i}]$. As $\alpha_i^q = \mathop{\rm argmin}_{\alpha\in\mathbb{R}}q_i(\alpha),$ and $\alpha_i^q \leq \alpha_i^*$,  we get
	\[
f_i(\alpha_i^q) \leq q_i(\alpha_i^q) = \min_{\alpha \in \mathbb R} q_i(\alpha) = f_i(0) - \frac{1}{2L \|s_i\|_{\mathcal V}^2}|f_i'(0)|^2 = f_i(0) - \frac{L}{2} \| \alpha_i^q s_i\|^2_{\mathcal{V}}.
	\]
In the last step, we have used the definition of $\alpha_i^q$ and this completes the proof.
\end{proof}

Since $\alpha_i^q$ has the same lower bound as $\alpha_i^*$, we can derive the upper bound in exactly the same way as the proof of Theorem~\ref{thm:upper-FASD}, only replacing $\varepsilon_i = \alpha_i^* s_i$ by $\alpha_i^q s_i$ and using the lower bound from Lemma~\ref{lm:alphai}. Thus, we only state the theorem below, and the proof is omitted.

	\begin{theorem}
Let $u^k$ be the $k$-th iteration and $u^{k+1} = {\rm FASD-ALS}(u^{k})$. Suppose the space decomposition satisfies (SS1) and (SS2), the energy $E$ satisfies (E1) -- (E2), and $E_i$ satisfies (E3) -- (E4). Then we have the upper bound
$$
E(u^{k+1}) - E(u) \leq C_U\sum _{i=1}^N\| \alpha_i^q s_i\|_{\mathcal V}^2, $$
where $C_U:=  C_A^2 \left[  C_S+ L  \left( 1 + \max_i \{  L_i/\mu_i \} \right)  \right]^2 /(2 \mu)$.
	\end{theorem}

We summarize the linear convergence result below.
\begin{corollary}
Let $u^k$ be the $k$-th iteration and $u^{k+1} = {\rm FASD-ALS}(u^{k})$. Suppose that the space decomposition satisfies assumptions (SS1) and (SS2), the energy $E$ satisfies assumption (E1) -- (E2), and the energy $E_i$ satisfies assumption (E3) -- (E4), then we have
	\[
E(u^{k+1}) - E(u) \leq \rho ( E(u^{k}) - E(u)),
	\]
with
	\[
\rho = \frac{C_A^2 \left[  C_S+ L  \left( 1 + \max_i \{  L_i/\mu_i \} \right)  \right]^2}{C_A^2 \left[  C_S+ L  \left( 1 + \max_i \{  L_i/\mu_i \} \right)  \right]^2 + L\mu}.
	\]
\end{corollary}

The Lipschitz constant $L$ is used in the step size $\alpha_i$ which can be replaced by a local Lipschitz constant for the scalar function $f_i(\alpha)$, for $\alpha \in (0, \alpha_i^*)$ and popular line search algorithms can be used. 

%
%

	\section{Original FAS Method: FASD Without Line Search}
	\label{sec:Original-FAS}

Notice that the original FAS does not have the last line search step. Traditional FAS, listed as Algorithm~\ref{alg:FAS-no-line}, applies the subspace correction via
	\[
v_{i} := v_{i-1} +\LC{ \alpha_i^{\rm FAS}} s_i, \quad \alpha_i^{\rm FAS} = 1.
	\]
	
		\begin{algorithm}[htp!]
\TitleOfAlgo{ $u^{k+1} = {\rm FAS}(u^k)$}
$v_0 = u^k$
	\;
\For{$i = 1 : N$}{
Compute the subspace \emph{$\tau$-perturbation}: let $\xi_i = Q_iv_{i-1}$ and 
	\begin{equation}
 \tau_i  := E'_{i}(\xi_i) - R_iE'(v_{i-1}) \in\mathcal{V}_i' ;
	\end{equation}
	\\
Solve the subspace residual problem: Find $\eta_i\in \mathcal{V}_i$, such that
	\begin{equation}
	\label{eq:residual-2}
\langle E'_{i}( \eta_i ), w \rangle =  \langle \tau_i, w \rangle, \; \forall w \in \mathcal{V}_i ;
	\end{equation}
	\\
Compute the correction
	\begin{equation}
s_i := \eta_i- \xi_i \in \mathcal{V}_i ;
	\end{equation}
	\\
Apply the subspace correction:
	\begin{equation}
v_{i} := v_{i-1} +  s_i.
	\end{equation}
}
$u^{k+1} := v_{N}$
	\;
	\medskip
\caption{Traditional FAS: FASD with no line search.} 
	\label{alg:FAS-no-line}
	\end{algorithm}

 \LC{Previously, our choice of step size was motivated by the choice of step size in the gradient descent method~\cite{Nesterov:2013Introductory}. We shall prove $\alpha_i^{\rm FAS} = 1$ is also allowed -- that is to say, it leads to a convergent algorithm -- provided that the following approximation property is satisfied.}
 
	\begin{itemize} 
	\item[(AP)] 
Both $E$ and $E_i$ are twice Fr\'{e}chet differentiable. Furthermore, there exists a constant $0<\epsilon < \mu/2$ so that for all \LC{$w\in {\mathcal B}, \eta_i \in \mathcal V_i$} and all $u_i, v_i \in \mathcal{V}_i$ 
	\[
|  \langle E''(w+\eta_i) u_i, v_i \rangle - \langle E_i''(Q_i w+\eta_i) u_i, v_i \rangle | \leq \epsilon \| u_i \|_{\mathcal{V}} \| v_i \|_{\mathcal{V}}.
	\]
	\end{itemize}
	
	
For quadratic energy, $R_i E'' I_i$ is the coarse matrix on $\mathcal V_i$ formed by the triple product, \LC{via the so-called Galerkin method}, and $E_i''$ is the matrix obtained using the bilinear form associated to the local energy $E_i$. They should be close in a certain norm.

\LC{The original FAS is to choose $E_i = E|_{\mathcal V_i}$ so that is $E_i'' = E''$ on $\mathcal{V}_i\times \mathcal{V}_i$. Assume furthermore $E''$ is also Lipschitz continuous. Then (AP) can be verified 
\begin{equation} \label{eqn:verify-AP}
\| E''(w+\eta_i) - E''(Q_i w+\eta_i) \| \leq C \|w - Q_i w\| .
\end{equation}
	Note that in this case the local problem $E'(\eta_i) = \tau_i$ is cheaper than solving the Euler equation $E'(v_{i-1} + e_i) = 0$ in SSO.
}

	\LC{\begin{lemma}
Assume the energy $E$ satisfies the assumptions (E1) and (E2), and the approximation assumption (AP) holds. Then, $E_i'$ satisfies the Lipschitz condition and strongly convexity condition as follows,
	\[
(\mu-\epsilon) \| v-w \|_{\mathcal{V}}^2 \le \langle E'_i(v) - E'_i(w), v-w \rangle  \le  (L+\epsilon) \| v-w \|_{\mathcal{V}}^2, \quad \forall \ v, w\in \mathcal B\cap \mathcal V_i.
	\]
	\end{lemma}}

	\LC{\begin{proof}
For any $v,w \in \mathcal V_i$, by Taylor's Theorem with integral reminder, we have
	\begin{align*}
\langle E'_i(v) & - E'_i(w), v-w \rangle 
	\\
&= \int_0^1 \langle E''_i(z(t))(v-w), v-w \rangle \mathrm{d}t
	\\
&= \int_0^1 \langle E''(z(t))(v-w), v-w \rangle \mathrm{d}t
	\\
& \quad + \int_0^1 \langle E''_i(z(t))(v-w), v-w \rangle - \langle E''(z(t))(v-w), v-w \rangle \mathrm{d}t 
	\\
& = \langle E'(v) - E'(w), v-w \rangle 
	\\
& \quad + \int_0^1 \langle E''_i(z(t))(v-w), v-w \rangle - \langle E''(z(t))(v-w), v-w \rangle \mathrm{d}t,
	\end{align*}
where $z(t) = tv + (1-t)w \in \mathcal V_i$, and thus $Q_iz = z$. When $v, w\in \mathcal B\cap \mathcal V_i$, using assumptions (E2) and (AP), we have
	\[
\langle E'_i(v) - E'_i(w), v-w \rangle \leq L \| v-w \|_{\mathcal{V}}^2 + \epsilon \| v-w \|_{\mathcal{V}}^2	= (L+\epsilon) \| v-w \|_{\mathcal{V}}^2.
	\]
On the other hand,  when $v, w\in \mathcal B \cap \mathcal V_i$, using assumptions (E1) and (AP), we have
	\[
\langle E'_i(v) - E'_i(w), v-w \rangle \geq \mu \| v-w \|_{\mathcal{V}}^2 - \epsilon \| v-w \|_{\mathcal{V}}^2	= (\mu-\epsilon) \| v-w \|_{\mathcal{V}}^2.
	\]
	\end{proof}}

	\begin{theorem} \label{thm:conv-FAS}
Let $u^k$ be the $k$-th iteration and $u^{k+1} = {\rm FAS}(u^{k})$, as in Algorithm~\ref{alg:FAS-no-line}, with local step size $\alpha_i^{\rm FAS} = 1$. Suppose that $E$ satisfies assumption (E1) and the approximation assumption (AP) holds with $\epsilon < \mu/2$. Then,  we have
	\begin{equation*}
 E(u^{k}) - E(u^{k+1}) \geq  C_L \sum _{i=1}^N \|s_i\|^2_{\mathcal V}, \quad C_L = \left( \frac{\mu}{2} - \epsilon \right).
	\end{equation*}
	\end{theorem}

	\begin{proof}
Recall that $\xi_i = Q_iv_{i-1}$ and $Q_i s_i = s_i$. Using equation~\eqref{eq:residual-2} and Taylor's theorem with integral remainder, we first estimate $|\langle E'(v_{i}), s_i \rangle |$ by
	\begin{align*}
|\langle E'(v_{i}), s_i \rangle | & = \left | \langle E'(v_{i-1} + s_i), s_i \rangle -  \langle E'(v_{i-1}), s_i \rangle  - [\langle E_i'(\xi_i + s_i), s_i \rangle  - \langle E'_i(\xi_i), s_i \rangle]  \right |
	\\
& =  \left |\int_0^1 \langle E''(y(t))\ s_i, s_i \rangle -  \langle E_i''(Q_iy(t))\ s_i, s_i \rangle \dd t \right |
	\\
& \LC{\le} \int_0^1 \left| \langle E''(y(t))\ s_i, s_i \rangle -  \langle E_i''(Q_iy(t))\ s_i, s_i \rangle \right | \dd t 
	\\
& \leq  \epsilon \|s_i\|^2_{\mathcal V},
	\end{align*}
\LC{
where $ y(t) := (1-t)v_{i-1} + t(v_{i-1} + s_i) = v_{i-1} + ts_i$. Note that $v_{i-1} \in \mathcal{B}$ and $s_i \in \mathcal{V}_i$ which allows us to use Assumption (AP) in the last step. 
}

Using assumption (E1) -- specifically estimate \eqref{ineq:strict-convexity} of Theorem~\ref{thm:convex&coercive} -- we get
	\begin{equation}
	\label{eq:lowerbound1}
E(v_{i-1}) - E(v_{i-1} + s_i) \geq -\langle E'(v_{i-1} + s_i), s_i \rangle + \frac{\mu}{2} \| s_i \|^2_{\mathcal{V}} \ge  \left( \frac{\mu}{2} - \epsilon \right) \| s_i\|^2_{\mathcal{V}}.
	\end{equation}
	\end{proof}

\LC{The upper bound for FAS (where $\alpha_i^{\rm FAS} = 1$) is easy, as there is now no need to have a lower bound of the step size.}

	\begin{theorem}
Let $u^k$ be the $k$-th iteration and $u^{k+1} = {\rm FAS}(u^{k})$ with local step size $\alpha_i = 1$. Suppose the space decomposition satisfies (SS1) and (SS2), the energy $E$ satisfies (E1) -- (E2), and assumption (AP) holds. Then we have the upper bound
	\[
E(u^{k+1}) - E(u) \leq C_U \sum _{i=1}^N\|s_i\|_{\mathcal V}^2,
	\]
where $C_U = C_A^2 (C_S  + \epsilon)^2/(2\mu)$.
	\end{theorem}
	\begin{proof}
For any $w \in \mathcal{V}$, we choose a stable decomposition $w = \sum_{i=1}^N w_i$, then
	\begin{align*}
\langle E'(u^{k+1}), w \rangle &= \sum_{i=1}^N \langle E'(u^{k+1}), w_i \rangle
	\\
& = \sum_{i=1}^N \langle E'(u^{k+1}) - E'(v_i), w_i \rangle + \sum_{i=1}^N\langle E'(v_i), w_i \rangle
	\\
& = {\rm I}_1 + {\rm I}_2
	\end{align*}

The first term is bounded as before. Therefore,
	\begin{equation*}
{\rm I}_1 \leq C_S C_A \left (\sum_{i=1}^N \|s_i\|_{\mathcal V}^2\right )^{1/2}\|w\|_{\mathcal V}.
	\end{equation*}

For the second term, we insert $\tau_i - E_i'(\xi_i + s_i) = - E'(v_{i-1}) + E_i'(\xi_i) - E_i'(\xi_i + s_i)$, which is zero in $\mathcal V_i'$, and use \LC{Taylor's Theorem with integral remainder, followed by assumption (AP),} to get
	\begin{align*}
{\rm I}_2 &=  \sum _{i=1}^N \langle E'(v_i) - E'(v_{i-1}) - \left [E_i'(\xi_i + s_i) - E_i'(\xi_i) \right ], w_i \rangle
	\\
& \leq \epsilon  \sum _{i=1}^N  \| s_i\|_{\mathcal V}\|w_i\|_{\mathcal V}
	\\
&\leq \epsilon C_A \left (\sum_{i=1}^N \| s_i\|_{\mathcal V}^2\right )^{1/2}\|w\|_{\mathcal V}.
	\end{align*}
	\end{proof}

	\begin{corollary} \label{coro:FAS-convergence}
Let $u^k$ be the $k$-th iteration and $u^{k+1} = {\rm FAS}(u^{k})$. Suppose that the space decomposition satisfies assumptions (SS1) and (SS2), the energy $E$ satisfies assumption (E1) -- (E2), and the energy $E_i$ satisfies assumption (AP) with $\epsilon < \mu/2$, then we have
	\[
E(u^{k+1}) - E(u) \leq \rho ( E(u^{k}) - E(u)),
	\]
with
	\[
\rho = \frac{(C_S+ \epsilon)^2C_A^2}{(C_S+\epsilon)^2C_A^2 + \mu (\mu - 2 \epsilon)}.
	\]
	\end{corollary}


\section{Application and Numerical Experiments} \label{sec:numerics}
In this section we shall apply our theory to a model nonlinear problem with polynomial nonlinearity and provide numerical examples to illustrate the efficiency of a variant of FAS (Algorithm~\ref{alg:FAS-no-line}) with a local quadratic energy. 

	\subsection{A Model Nonlinear Problem}

Suppose that $\Omega\subset\mathbb{R}^d$, $d = 2,3$, is a star-shaped polytope, i.e. a polygon in 2-D or a polyhedron in 3-D. Suppose that $2\le p<\infty$, when $d = 2$, and $2\le p\le 6$, when $d=3$. We consider the following problem: given $f\in L^2(\Omega)$, find $u \in H_0^1(\Omega)$ such that
	\begin{equation}
	\label{eqn:pnonlinear}
\iprd{|u|^{p-2} u}{  \xi} +\varepsilon^2 \iprd{\nabla u}{\nabla \xi}  = \iprd{f}{\xi} , \quad  \forall \ \xi \in H_0^1(\Omega),
	\end{equation}
where $\varepsilon>0$ is parameter. One can show that the unique solution of \eqref{eqn:pnonlinear} is the unique minimizer of a certain strictly convex energy. 

\LC{	
	\begin{theorem}
Suppose that $\Omega\subset\mathbb{R}^d$, $d = 2,3$, is a star-shaped polytope, i.e. a polygon in 2-D or a polyhedron in 3-D. Suppose that $2\le p<\infty$, when $d = 2$, and $2\le p\le 6$, when $d=3$. For any $\nu \in H_0^1(\Omega)$, define the energy
	\begin{equation}
	\label{eqn:eng4c}
E(\nu) :=  \frac{1}{p} \norm{ \nu}{L^p}^p +\frac{\varepsilon^2}{2} \| \nabla \nu \|^2 - \iprd{f}{\nu} .
	\end{equation}
	The energy functional $E$ defined in \eqref{eqn:eng4c}  is twice Fr\'{e}chet differentiable and satisfies assumptions (E1) and (E2) with respect to the space $\mathcal V = H_0^1(\Omega)$, equipped with the norm $\| \nabla v\|$, for $v\in \mathcal V$. 
%
Therefore $E$ has a unique global minimizer in $H_0^1(\Omega)$. Furthermore, $u\in H_0^1(\Omega)$ is the unique minimizer of \eqref{eqn:eng4c} iff it is the solution of \eqref{eqn:pnonlinear}.
	\end{theorem}
	}
	\begin{proof}
	We verify that $E$ satisfies our assumptions. The first Fr\'{e}chet derivative of $E$ at a point $\nu$ may be calculated as follows: for any $\xi\in H_0^1(\Omega)$, 
	\[
\left. \frac{\dd}{\dd t} E(\nu + t \xi)\right|_{t=0} = \langle E'(\nu), \xi\rangle =  \iprd{|\nu|^{p-2}\nu}{ \xi} + \varepsilon^2\iprd{\nabla \nu}{\nabla \xi}  -\iprd{f}{\xi} .
	\]

	The second Fr\'{e}chet derivative exists for $p\ge 2$ and is a continuous bilinear operator. Given a fixed  $\nu\in H_0^1(\Omega)$, the action of the second variation on the  arbitrary pair $(\xi,\eta)\in H_0^1(\Omega)\times H_0^1(\Omega)$ is given by
	\[
\langle E''(\nu)\xi,\eta \rangle  =  (p-1)\iprd{|\nu|^{p-2}\xi}{\eta} +\varepsilon^2\iprd{\nabla \xi}{\nabla \eta}.
	\]


Without loss of generality, we choose $u_0 = 0$, so that $E(u_0) = 0$. Recall that $\mathcal{B}= \left\{v \in  \mathcal V  \ \middle| \ E(v)\le E(u_0)\right\}$. Observe that $\mathcal{B}$ is convex, since $E$ is convex. For $v\in \mathcal B$, $E(v) \le 0$, and we have
	\[
\frac{1}{p} \norm{ v}{L^p}^p +\frac{\varepsilon^2}{2} \| \nabla v \|^2 \leq (f, v) \leq \|f\|\|v\| \leq C_0(\varepsilon,C_{p,\Omega})\|f\|^2 + \frac{\varepsilon^2}{4}\|\nabla v\|^2,
	\]
where $C_{p,\Omega}>0$ is the constant in the Poincare inequality:
	\[
\nrm{v} \le C_{p,\Omega} \nrm{\nabla v}, \quad \forall \ v\in H_0^1(\Omega).
	\]
Thus, for $v\in \mathcal B$, the follow norms are bounded: 
	\begin{equation}
	\label{eqn:LpH1bound}
\norm{ v}{L^p} + \| \nabla v\| \leq C_1 = C_1(u_0,\varepsilon,p,f).
	\end{equation}

By the mean value theorem, there exists a $z = t v + (1-t)w$, for some $t\in [0,1]$, such that
	\[
\langle E'(w), \xi\rangle - \langle E'(v), \xi\rangle = \langle E''(z)\xi, w-v \rangle, \quad \forall\  \xi \in H_0^1(\Omega).
	\]
If $w,v\in \mathcal{B}$, then, since $\mathcal{B}$ is convex, $z\in\mathcal{B}$.  By \eqref{eqn:LpH1bound} $\nrm{ z}_{L^p} \le C_1$. Using H{\"o}lder's inequality, we have
	\begin{align*}
\left | \langle E''(\nu)\xi,\eta \rangle \right | &\le (p-1)\norm{\nu}{L^p}^{p-2}\norm{\xi}{L^p}\norm{\eta}{L^p} + \varepsilon^2\nrm{\nabla \xi } \cdot \nrm{\nabla \eta}	
\\
& \le \left [ (p-1) C_{p,\Omega}^2 \norm{\nu}{L^p}^{p-2} + \varepsilon^2  \right ] \nrm{\nabla \xi } \cdot \nrm{\nabla \eta}.
	\end{align*}
Therefore,
	\begin{align*}
\big| \langle E'(w), \xi\rangle - &  \langle E'(v), \xi\rangle \big| 
= \ \left| \langle E''(z)\xi,w-v \rangle \right| 
	\\
\le & \left [ (p-1) C^2_{p,\Omega} C_1^{p-2} + \varepsilon^2 \right ]\nrm{\nabla \xi } \cdot \nrm{\nabla (w-v)}.
	\end{align*}
Namely (E2) holds with $L := (p-1) C^2_{p,\Omega} C_1^{p-2} + \varepsilon^2$.

To see that $E$ is uniformly elliptic, for any $w,v\in\mathcal{V}$, there is an $\eta\in\mathcal{V}$, 
	\begin{align*}
\langle E'(w) - E'(v), w-v \rangle & = \langle E''(z) (w - v) , w-v \rangle, 
	\\
& = (p-1)\iprd{|\nu|^{p-2}(w - v)}{w-v} +\varepsilon^2\iprd{\nabla (w - v)}{\nabla (w - v)}
	\\
& \ge \varepsilon^2\nrm{\nabla (w - v)}^2.
	\end{align*}
(E1) holds with $\mu= \varepsilon^2$.


It follows that there is a unique global minimizer of the energy \eqref{eqn:eng4c}:
	\[
u := \mathop{\rm argmin}_{\nu\in H_0^1(\Omega)} E(\nu) .
	\]
Consequently, there is a unique solution to the Euler problem which is equation \eqref{eqn:pnonlinear}.
	\end{proof}


Now, suppose that $\Omega\subset\mathbb{R}^2$ is a polygonal domain and $\TauH$ is a conforming triangulation of $\Omega$. Let $\Tauh
$ be the triangulation obtained by quadri-secting $\TauH$. Specifically, if $K_i \in \Tauh$ is one of the four daughter triangles ($i = 1,\cdots, 4$)  obtained by quadri-secting $K\in \TauH$ -- that is by connecting the midpoints of $K$ -- then $h_{K_i} = H_K/2$, $i = 1,\cdots, 4$. A family of meshes constructed in this way is known to be globally quasi-uniform.

Define
	\[
S_h := \left\{ v\in C(\Omega)\cap H_0^1(\Omega) \middle| \left. v\right|_K \in \mathcal{P}_1(K), \  \forall K\in \Tauh \right\}	.
	\]
With a similar definition for $S_H$. Then, $S_H \subset S_h$, and the containment is proper. 

We shall consider the minimization of energy $E$ restricted to $S_h$ which is a subspace of $H_0^1(\Omega)$
	\[
\min_{v\in S_h} E(v),
	\]
and thus now $\mathcal V = S_h$ with norm $|v|_1 = \|\nabla v\|$. Notice that (E1) and (E2) still hold, as $S_h \subset H_0^1(\Omega)$.

Next we give a two-level space decomposition of $\mathcal V$ as follows.
Let $\mathcal{N} = \left\{{\bf x}_i\right\}_{i = 1}^N\subset  \mathbb{R}^2$ be the set of \emph{interior} nodes of $\Tauh$ and define the Lagrange nodal basis
	\[
B_h = \left\{ \psi_i \in S_h \  \middle| \ \psi_i({\bf x}_j) = \delta_{i,j} , \  1\le i,j\le N\right\}.
	\]
$B_h$ is a \emph{bona fide} basis for $S_h$, and we may use the following decomposition 
	\begin{equation}
\mathcal V = \sum_{i=0}^N \mathcal V_i = S_h,
	\label{eq:FEM-space-decomp}
	\end{equation}
where $\mathcal{V}_0 = S_H$, $\mathcal{V}_i = {\rm span}(\left\{ \psi_i\right\})$, $1\le i \le N$. \LC{(Note that we give the coarse space the index $0$.) }

The fact that this forms a stable decomposition is well known, i.e., Assumption (SS1) holds. 
\LC{	
\begin{lemma}
The decomposition of the finite element space $S_h$ described in \eqref{eq:FEM-space-decomp} satisfies Assumption (SS1).
	\end{lemma}
	}
	\begin{proof}
Let $Q_H:L^2(\Omega)\to S_H$ be the $L^2$-projection into $S_H$:
	\[
(Q_H v, w) = (v,w), \quad \forall \ w\in S_H. 	
	\]
For any $v\in S_h$, let $\tilde v = (I - Q_H)v \in S_h$ denote the error, and suppose that $\tilde v = \sum_{i=1}^N \tilde v_i$ is the nodal decomposition of the error in $S_h$. By the standard approximation property of $Q_H$ on quasi-uniform grids, an inverse inequality, and the stability of nodal decompositions in the $L^2$-norm, we have 
	\[
\sum_{i=1}^N |\tilde v_i|_1^2 \le C \sum_{i=1}^N h^{-2}\|\tilde v_i\|^2 \le C h^{-2}\|\tilde v\|^2 \le C |v|_1^2.
	\]
By the $H^1$-stability of $Q_H$ on quasi-uniform grids, we also have $|Q_H v|_1\lesssim |v|_1$. In conclusion, Assumption (SS1) holds \LC{if, for $v\in S_h$, we use the decomposition
	\[
v = Q_H v + (v- Q_H v) = Q_H v + \sum_{i=1}^N \tilde v_i.
	\]}
	\end{proof}

\LC{
	\begin{lemma}
Let $E$ be defined as in \eqref{eqn:eng4c}, and let $\mathcal{V} = S_h$ be decomposed into subspaces  as in \eqref{eq:FEM-space-decomp}. Then Assumption (SS2) holds.
	\end{lemma}
	}
	\begin{proof}
Suppose that $w_{i,j}\in \mathcal B$,  $u_i \in \mathcal V_i$, $v_j\in \mathcal V_j$, with $w_{i,j} + u_i \in \mathcal{B}$. By Taylor's theorem,
	\begin{align*}
\sum_{i=0}^{N}\sum_{j=i+1}^{N} \langle  E'(w_{i,j}+u_i) - & \, E'(w_{i,j}) , v_j \rangle 
 	\\
& =   \sum_{i=0}^{N}\sum_{j=i+1}^{N}  \langle E''(z_{i,j})v_j, u_i\rangle  
	\\
& \le \sum_{i=0}^{N}\sum_{j=i+1}^{N} \left| (p-1)\iprd{|z_{i,j}|^{p-2} u_i}{v_j} +\varepsilon^2\iprd{\nabla u_i}{\nabla v_j} \right| ,
	\end{align*}
for some $z_{i,j}\in \mathcal{B}$ between $w_{i,j}\in\mathcal{B}$ and $w_{i,j} + u_i\in\mathcal{B}$, which satisfies the bound \eqref{eqn:LpH1bound}. The functions $u_i$, $1\le i,j\le N$, are local, though $u_0$ may have global support. The support of $v_i$, $1\le i\le N$, denoted $S_i$, is exactly equal to the union of those triangles that have the node ${\bf x}_i$ as a vertex. Define
	\[
\mathcal N(i) := \left\{j >i\ \middle| \  S_j\cap S_i \neq \emptyset \right\}.
	\]
Observe that $\#\left( \mathcal N(i)\right)$ is bounded by an integer that is much smaller than $N$. We have, using the continuous and discrete Cauchy-Schwartz inequalities, 
	\begin{align*}
\sum_{i=0}^{N} \sum_{j=i+1}^N \left(\nabla u_i,\nabla v_j\right)  & = \sum_{i=0}^{N} \sum_{j\in \mathcal N(i)} \left(\nabla u_i,\nabla v_j\right)_{S_i\cap S_j} 
	\\
& \le   \sum_{i =0}^N\sum_{j\in \mathcal N(i)}  \nrm{\nabla u_i}_{S_i\cap S_j}\nrm{\nabla v_j}_{S_i\cap S_j} 
	\\
& \le  \left(\sum_{i =0}^N\sum_{j\in \mathcal N(i)} \nrm{\nabla u_i}_{S_i\cap S_j}^2 \right)^\frac{1}{2} \left( \sum_{i =0}^N\sum_{j\in \mathcal N(i)}  \nrm{\nabla v_j}_{S_i\cap S_j}^2\right)^\frac{1}{2}
	\\
& \le \left( C_{\mathcal{T}} \sum_{i =0}^N  \nrm{\nabla u_i}^2 \right)^\frac{1}{2} \left( C_{\mathcal{T}} \sum_{j=0}^N  \nrm{\nabla v_j}^2\right)^\frac{1}{2},
	\end{align*}
where $C_{\mathcal{T}}>0$ is a mesh-structure-dependent parameter.  Since our mesh is shape regular and quasi-uniform, $C_{\mathcal{T}}$ is independent of $N$ and $h$. 

Similarly, 
	\begin{align*}
\sum_{i=0}^{N}\sum_{j\in \mathcal N(i)} \iprd{|z_{i,j}|^{p-2} u_i}{v_j} & = \sum_{i=0}^{N}\sum_{j\in \mathcal N(i)} \iprd{|z_{i,j}|^{p-2} u_i}{v_j}_{S_i\cap S_j}
	\\
& \le \sum_{i=0}^{N}\sum_{j\in \mathcal N(i)} \nrm{z_{i,j}}_{L^p(S_i\cap S_j)}^{p-2} \nrm{u_i}_{L^p(S_i\cap S_j)}\nrm{v_j}_{L^p(S_i\cap S_j)}
	\\
& \le \sum_{i=0}^{N}\sum_{j\in \mathcal N(i)} \nrm{z_{i,j}}_{L^p(\Omega)}^{p-2} \nrm{u_i}_{L^p(S_i\cap S_j)}\nrm{v_j}_{L^p(S_i\cap S_j)}
	\\
& \le \sum_{i=0}^{N}\sum_{j\in \mathcal N(i)} C_1^{p-2} \nrm{u_i}_{L^p(S_i\cap S_j)}\nrm{v_j}_{L^p(S_i\cap S_j)}
	\\
& \le C_1^{p-2} \left( \sum_{i=0}^{N}\sum_{j\in \mathcal N(i)}  \nrm{u_i}^2_{L^p(S_i\cap S_j)}\right)^{\frac{1}{2}}
	\\
& \quad \times  \left( \sum_{i=0}^{N}\sum_{j\in \mathcal N(i)}  \nrm{v_j}^2_{L^p(S_i\cap S_j)}\right)^{\frac{1}{2}}
	\\
& \le C_1^{p-2} \left( C_{\mathcal{T}} \sum_{i=0}^{N}  \nrm{u_i}^2_{L^p}\right)^{\frac{1}{2}}\left(C_{\mathcal{T}}  \sum_{i=0}^{N} \nrm{v_i}^2_{L^p}\right)^{\frac{1}{2}}
	\\
& \le C_1^{p-2} C_{\mathcal{T}} \left(  \sum_{i=0}^{N} C_{p,\Omega}^2 \nrm{\nabla u_i}^2 \right)^{\frac{1}{2}}\left( \sum_{i=0}^{N}  C_{p,\Omega}^2  \nrm{\nabla v_i}^2_{L^p}\right)^{\frac{1}{2}}
	\\
& = C_1^{p-2} C_{\mathcal{T}} C_{p,\Omega}^2  \left(  \sum_{i=0}^{N} \nrm{\nabla u_i}^2 \right)^{\frac{1}{2}}\left( \sum_{i=0}^{N} \nrm{\nabla v_i}^2_{L^p}\right)^{\frac{1}{2}}.
	\end{align*}
Therefore, there is a $C_S>0$ such that 
	\[
\sum_{i=0}^N\sum_{j=i+1}^N \langle  E'(w_{i,j}+u_i) -  E'(w_{i,j}) , v_j \rangle 	\le C_S \left(\sum_{i = 0}^N \nrm{\nabla u_i}^2 \right)^\frac{1}{2}	\left( \sum_{j = 0}^N \nrm{\nabla v_j}^2\right)^\frac{1}{2}.
	\]
In particular, $C_S := L C_{\mathcal{T}}$. Assumption (SS2) holds.
	\end{proof}

We apply SSO to each subspace $\mathcal V_i$ on the fine level, which is equivalent to, according to Remark \ref{rm:FASD-SSO}, use  $E_i(\eta)  := E(v_{i-1}-Q_i v_{i-1} +\eta)$. On the coarse space, we use $E_H = E|_{\mathcal{V}_H}$. The Assumption (AP) can be verified by~\eqref{eqn:verify-AP} and standard approximation property of the projection $Q_H$, i.e., $\| w - Q_H w \| \leq C H \| w \|_1$.  We have $\epsilon = C H$ for this case and, therefore, the condition $\epsilon < \mu/2 = \varepsilon^2/2$ in Theorem~\ref{thm:conv-FAS} holds when $H$ is small enough. For finite element functions $w\in V_h$, when near the minimizer, we could expect $w\in H^{3/2-\delta}$ for any $0< \delta \ll 1$ and thus high order approximation $\| w - Q_H w \| \leq C H^{3/2-\delta} \| w \|_{3/2-\delta}$ may hold.


	\subsection{Numerical Examples} 
In this subsection, we present some numerical results for the nonlinear problems described in the previous two subsections to illustrate our theoretical results.   For both problems, we will use piece-wise linear finite elements define $S_h$, and we use different versions FAS to solve the discretized nonlinear equations.  Our algorithms are implemented in MATLAB based on the software package $i$FEM~\cite{Chen.L2008c}.  The numerical experiments are conducted on a System76 Galago with an Intel Core i7-8550U CPU and 32GB RAM.  


We mainly focus on three different implementations of FAS (Algorithm~\ref{alg:FAS-no-line}), based on different choices of space decomposition and local energy. The geometric multigrid setting is considered here, i.e., we have a set of uniformly refined meshes and nested linear finite element spaces $\mathcal{V}^1 \subset \mathcal{V}^2 \subset \cdots \subset \mathcal{V}^J$, where $\mathcal{V}^\ell = \text{span}\{ \phi_1^\ell, \phi_2^\ell, \cdots, \phi_{N_\ell}^{\ell}\}$, with $\phi_i^\ell$ being the $i^{\rm th}$ nodal linear finite element basis element on level $\ell$.
	\begin{enumerate}
	\item 
The first implementation is the original FAS.  We consider standard multilevel nodal based space decomposition~$\mathcal{V} = \sum_{\ell=1}^{J} \sum_{i=1}^{N_{\ell}} \LC{\text{span}\{ \phi_i^{\ell}\}}$ and the local energy $E_i$ is defined as the restriction of $E$ on the subspace $\text{span}\{ \phi_i^{\ell} \}$. Newton's method is used to solve the local nonlinear problem and we set the tolerance to be $10^{-10}$ and at most $100$ iterations are allowed (in general, less than $5$ iterations are needed for solving the local problems in all of our numerical tests). We use a small tolerance to make sure each local problem is solved exactly in order to be consistent with our theoretical analysis. 
	\item 
The second implementation is a simplified version of FAS based on Remark~\ref{rem:quadratic-Ei} and we refer to it as ``FASq1".  We again consider the multilevel nodal based space decomposition~$\mathcal{V} = \sum_{\ell=1}^{J} \sum_{i=1}^{N_{\ell}}\LC{\text{span}\{ \phi_i^{\ell}\}}$ but quadratic energy~$E_i$ defined as in~\eqref{linearEi} is used, which requires that we solve a linear system for each local correction.  \LC{In fact, since nodal based space decomposition is used here, we solve a scalar linear equation on each subspace.}
	\item 
The third implementation is a further simplified version and we refer it as ``FASq2". In this case, we use space decomposition~$\mathcal{V} = \sum_{\ell=1}^{J} \mathcal{V}^{\ell}$ and consider quadratic energy~\eqref{linearEi}. As mentioned in Remark~\ref{rem:quadratic-Ei}, this involves the Riesz map which can be computed by inverting an SPD matrix defined on $\mathcal{V}^\ell$. For our example, this is equivalent to solving a discrete Laplacian matrix on each level, which is still expensive.  Therefore, we solve the discrete Laplacian matrix approximately by just applying one step of symmetric Gauss-Seidel (SGS) method. \LC{This is because we use multilevel space decomposition here and SGS method is usually used as a smoother in multigrid methods for solving discrete Laplacian matrix. Of course, other types of iterative methods can also be used here, such as Richardson method and Jacobi method. For the sake of simplicity, we only consider SGS method here.}
	\end{enumerate}
In all of our numerical experiments, we use Newton's method to solve the nonlinear problem on the coarsest level. We use $10^{-10}$ as the tolerance and maximal number of iterations is $100$, which means that the coarse problem is solved exactly. Moreover, we use $\alpha_i = 1$ in the tests to make sure our implementation is simple and practical.  The overall stopping criterion of FAS is $10^{-10}$. 


	\begin{table}[htp]
\caption{Numerical results of FAS (varying $p$ and $\varepsilon$, fix $h = 1/64$)}  \label{tab:FAS}
	\begin{center}
		\vskip -8pt
		\resizebox{\linewidth}{!}{
		\begin{tabular}{|l||c|c|c|c|c|c|c|}
			\hline\hline
		FAS	& $\varepsilon^2 = 1$ & $\varepsilon^2 = 1/2$ & $\varepsilon^2 = 1/4$ & $\varepsilon^2 = 1/8$ & $\varepsilon^2 = 10^{-1}$ &   $\varepsilon^2 = 10^{-2}$	&   $\varepsilon^2 = 10^{-3}$ \\ \hline
			$p=4$	& 15 (0.195) & 15 (0.193) & 14 (0.189) & 14 (0.186) & 14 (0.186) & 12 (0.164) &  10 (0.133)	\\
			$p=5.5$	& 14 (0.195) & 14 (0.192) & 14 (0.189) & 14 (0.189) & 14 (0.189) & 12 (0.166) &  11 (0.162)	\\
			$p=6$	& 15 (0.195) & 15 (0.192) & 14 (0.190) & 14 (0.190) & 14 (0.189) & 13 (0.167) &  11 (0.167)	\\
			$p=8$	& 15 (0.196) & 15 (0.193) & 15 (0.192) & 14 (0.191) & 14 (0.190) & 13 (0.176) &  12 (0.173)	\\
			$p=10$	& 15 (0.198) & 15 (0.196) & 15 (0.194) & 15 (0.192) & 14 (0.191) & 13 (0.178) &  12 (0.170)	\\
			$p=20$	& 16 (0.216) & 16 (0.221) & 16 (0.210) & 15 (0.197) & 15 (0.194) & 14 (0.182) &  13 (0.178)	\\
			$p=40$	& 18 (0.267) & 18 (0.273) & 17 (0.248) & 16 (0.209) & 16 (0.204) & 14 (0.188) &  13 (0.180)	\\
			$p=80$	& 21 (0.333) & 21 (0.338) & 20 (0.304) & 18 (0.243) & 17 (0.226) & 15 (0.192) &  14 (0.200)	\\
			\hline\hline
		\end{tabular}
	}
	\end{center}
	\label{default}
	\end{table}%
		
	\vskip -16pt
	
	\begin{table}[htp]
\caption{Numerical results of FASq1 (varying $p$ and $\varepsilon$, fix $h = 1/64$)}
	\label{tab:FASq1}
	\begin{center}
				\vskip -8pt
		\resizebox{\linewidth}{!}{
		\begin{tabular}{|l||c|c|c|c|c|c|c|}
			\hline\hline
		FASq1	& $\varepsilon^2 = 1$ & $\varepsilon^2 = 1/2$ & $\varepsilon^2 = 1/4$ & $\varepsilon^2 = 1/8$ & $\varepsilon^2 = 10^{-1}$ &   $\varepsilon^2 = 10^{-2}$	&   $\varepsilon^2 = 10^{-3}$ \\ \hline
			$p=4$	& 15 (0.193) & 15 (0.189) & 14 (0.185) & 14 (0.180) & 13 (0.179) & 23 (0.331) &  -	\\
			$p=5.5$	& 15 (0.192) & 15 (0.189) & 14 (0.186) & 14 (0.184) & 14 (0.183) & - &  -	\\
			$p=6$	& 15 (0.192) & 15 (0.189) & 14 (0.187) & 14 (0.185) & 14 (0.183) & - &  -	\\
			$p=8$	& 15 (0.193) & 15 (0.190) & 14 (0.190) & 14 (0.191) & 14 (0.186) & - &  -	\\
			$p=10$	& 15 (0.195) & 15 (0.193) & 14 (0.191) & 14 (0.192) & 14 (0.187) & - &  -	\\
			$p=20$	& 16 (0.211) & 16 (0.215) & 16 (0.215) & 16 (0.216) & 16 (0.220) & - &  -	\\
			$p=40$	& 18 (0.260) & 18 (0.281) & 19 (0.298) & 21 (0.334) & 23 (0.367) & - &  -	\\
			$p=80$	& 21 (0.342) & 23 (0.383) & 25 (0.407) & 109 (0.844) & - & - & -	\\
			\hline\hline
		\end{tabular}
	}
	\end{center}
	\label{default}
	\end{table}%
		
		\vskip -16pt
	
	\begin{table}[htp]
\caption{Numerical results of FASq2 (varying $p$ and $\varepsilon$, fix $h = 1/64$)}
	\label{tab:FASq2}
	\begin{center}
				\vskip -8pt
		\resizebox{\linewidth}{!}{
		\begin{tabular}{|l||c|c|c|c|c|c|c|}
			\hline\hline
			FASq2	& $\varepsilon^2 = 1$ & $\varepsilon^2 = 1/2$ & $\varepsilon^2 = 1/4$ & $\varepsilon^2 = 1/8$ & $\varepsilon^2 = 10^{-1}$ &   $\varepsilon^2 = 10^{-2}$	&   $\varepsilon^2 = 10^{-3}$ \\ \hline
			$p=4$	& 14 (0.190) & 14 (0.187) & 14 (0.183) & 14 (0.181) & 14 (0.181) & - &  -	\\
			$p=5.5$	& 14 (0.189) & 14 (0.189) & 14 (0.183) & 14 (0.185) & 14 (0.187) & - &  -	\\
			$p=6$	& 14 (0.188) & 14 (0.186) & 14 (0.185) & 14 (0.188) & 14 (0.190) & - &  -	\\
			$p=8$	& 14 (0.190) & 14 (0.190) & 14 (0.188) & 14 (0.193) & 15 (0.196) & - &  -	\\
			$p=10$	& 15 (0.191) & 15 (0.191) & 15 (0.193) & 15 (0.199) & 15 (0.202) & - &  -	\\
			$p=20$	& 15 (0.211) & 16 (0.223) & 17 (0.239) & 18 (0.265) & 20 (0.290) & - &  -	\\
			$p=40$	& 18 (0.264) & 19 (0.300) & 21 (0.334) & 29 (0.452) & 49 (0.643) & - &  -	\\
			$p=80$	& 21 (0.350) & 24 (0.393) & 32 (0.504) & - & - & - & -	\\
			\hline\hline
		\end{tabular}
	}
	\end{center}
	\label{default}
	\end{table}%

	\begin{table}[htp]
\caption{Computational complexity comparison with $\varepsilon =1$ and $p=6$} 
\begin{center}
\vskip -8pt
\begin{tabular}{|c||c|c||c|c|}
\hline \hline 
  	& \multicolumn{2}{c||}{FAS} & \multicolumn{2}{c|}{FASq2}  \\ \hline
 $h$ &  \#iter &  CPU time &  \#iter &  CPU time \\ \hline
 $1/32$ 	& 15 & 1.65  	&  14  &  0.03  \\ 
 $1/64$ 	& 15 & 7.86  	&  14  &  0.05  \\ 
 $1/128$ 	& 16 & 45.60  	&  14  &  0.16  \\ 
 $1/256$ 	& 16 & 391.08   &  15  & 0.49 \\ 
 $1/512$ 	& 16 &  $>$1,000 	&  15  & 1.67 \\ 
 $1/1024$   & 16 &  $>$1,000  &  15  & 7.12 \\ 
 \hline \hline 
\end{tabular}
\end{center}
	\label{tab:poly}
	\end{table}%
	

In Table~\ref{tab:FAS}, \ref{tab:FASq1}, and~\ref{tab:FASq2}, we report the numerical results of FAS, FASq1, and FASq2, respectively. Here, we fix the finest mesh size $h = 1/64$ and the coarsest mesh size is $1/4$ but change $p$ and $\varepsilon$ to adjust the nonlinearity. In this case, bigger $p$ and/or smaller $\varepsilon$ lead to stronger nonlinearity. 

Number of iterations and convergence rates (in the parenthesis) are listed in Table~\ref{tab:FAS}, \ref{tab:FASq1}, and~\ref{tab:FASq2}. \LC{Notation ``-" means that the methods stagnates or diverges.}   As we can see, FAS is the most robust one and converges for all the choices of our parameters. The number of iterations are quite stable, ranging from $10-21$ iterations, and the convergence rate is about $0.2$.  \LC{This is consistent with our theoretical results presented in Section~\ref{sec:Original-FAS}. For FAS, the local energy $E_i$ is defined as the restriction of $E$ on the subspace. Then Assumption (AP) holds with $\epsilon < \mu/2$.  Therefore, according to Corollary~\ref{coro:FAS-convergence}, FAS converges robustly.} For FASq1 and FASq2, both implementations perform well when $p$ is relatively small and/or $\varepsilon$ is relatively big.  We can clearly see that the number of iterations grows when $p$ gets larger or $\varepsilon$ gets smaller.  Both implementations fail to converge when nonlinearity is strong, while FASq1 seems to be slightly more robust than FASq2 since it converges for slight larger set of parameters.  \LC{This observation also consists with Corollary~\ref{coro:FAS-convergence}. For both FASq1 and FASq2, the local energy $E_i$ is the quadratic energy~\eqref{linearEi}. When $p$ is relatively small and/or $\varepsilon$ is relatively big, the nonlinearity of the model problem is relatively weak and quadratic energy provides a good approximation in the sense that Assumption (AP) holds with $\epsilon < \mu/2$.  According to Corollary~\ref{coro:FAS-convergence}, the methods should converge.  However, when $p$ gets larger and/or $\varepsilon$ gets smaller, the problem becomes more nonlinear and quadratic energy is not a good approximation of the original energy $E$ any more.  Then Assumption (AP) does not hold with $\epsilon < \mu/2$ and, according to Corollary~\ref{coro:FAS-convergence}, the method may not converge.} Although FASq1 and FASq2 might not converge for strongly nonlinear problem, the advantage of using quadratic energy on local subspaces is that we only need to solve linear problems locally, which could save computational cost considerably.  

Next, we compare the CPU time of FAS and FASq2. The reason we choose FASq2 to compare is that FASq2 only involves symmetric Gauss-Seidel smoother on each level, which basically has the same cost as the multigrid method for solving linear problems. This could dramatically improve the computational complexity for solving our model problem~\eqref{eqn:pnonlinear} and the results are shown in Table~\ref{tab:poly}.


In Table~\ref{tab:poly}, we fix $\varepsilon =1$ and $p=6$ and change $h$. As we can see, for these choices of $p$ and $\varepsilon$, the quadratic energy provides a good approximation of the global energy restricted to the subspace, therefore, the number of iterations of FASq2 is similar with the number of iterations of FAS and remain robust with respect to the mesh size $h$.  The CPU time of FAS grows faster than linear, which is due to the inefficiency of large \mc{for} loops in MATLAB.

  In contrast, FASq2 is significantly faster than FAS and scales linearly.  This demonstrates that, when nonlinearity is mild, we can use a simple quadratic energy and save computational cost. 

On the other hand, we want to point out that FAS is more robust than FASq2 as shown before.  We have also tested the quadratic energy defined by the Hessian at the previous iteration c.f., \eqref{NewtonEi}, which is more or less equivalent to using one approximated Newton's iteration,  and the results are similar. Therefore, in practice, we should consider the trade-off between robustness and efficiency in order to decide which kind of local energy should be used on each subspace.


	\section*{Acknowledgments}
The authors wish to thank the anonymous referees for their incisive comments, which helped to improve the paper greatly. The authors also thank Hao Luo, Jea-Hyun Park, and Abner Salgado for carefully reading the revised manuscript and for suggesting several improvements.

	\bibliographystyle{plain}
	\bibliography{FASD.bib}
	\end{document}